\documentclass{amsart}
\usepackage{amssymb}
\usepackage{mathrsfs}
\usepackage{stmaryrd}
\usepackage{bbm}
\usepackage{color}
\usepackage{oldgerm}
\usepackage[english]{babel}
\usepackage[T1]{fontenc}
\usepackage[latin1]{inputenc}
\usepackage[all]{xy}
\usepackage{hyperref}
\usepackage[all]{xy}
\usepackage{extarrows}
\usepackage{tikz-cd}
\usepackage{enumerate}

\newtheorem{theorem}[subsection]{Theorem}
\newtheorem{proposition}[subsection]{Proposition}

\newtheorem{lemma}[subsection]{Lemma}

\theoremstyle{definition}

\newtheorem{definition}[subsection]{Definition}

\newtheorem{remark}[subsection]{Remark}

\numberwithin{equation}{subsection}

\usepackage{amsmath}
\usepackage{amsthm}
\usepackage{mathrsfs}
\usepackage{amsfonts}

\begin{document}

\title {The stable GKZ hypergeometric $\mathcal D$-module}

\author{Lei Fu}
\address{Yau Mathematical Sciences Center, Tsinghua University, Beijing 100084, P. R. China}
\email{leifu@tsinghua.edu.cn}

\date{}
\maketitle

\begin{abstract}
For an $(n\times N)$-matrix $A$  of rank $n$ with integer entries, Gelfand, Kapranov and
Zelevinsky introduce a system of differential equations, called the   
$A$-hypergeometric system. We define the stable GKZ hypergeometric $\mathcal D$-module 
using cohomological functors, which is closely related to the $A$-hypergeometric $\mathcal D$-module
and the $\mathcal D$-module underlying the 
better behaved GKZ system introduced by Borisov and Horja. 
We prove the stable GKZ hypergeometric $\mathcal D$-module is holonomic
and is an integrable connection of rank $n!\mathrm{vol}(\Delta_\infty)$
on the Zariski open subset parametrizing nondegenerate Laurent polynomials, where $\Delta_\infty$ 
is the Newton polytope at $\infty$.


\end{abstract}

\section*{Introduction}
\subsection{The $A$-hypergeometric $\mathcal D$-module}
Let $$A=\left(\begin{array}{ccc} w_{11}&\cdots&w_{1N}\\
\vdots&&\vdots\\
w_{n1}&\cdots&w_{nN}\end{array} \right)$$ be an $(n\times N)$-matrix
of rank $n$ with integer entries. Denote the column vectors of $A$
by ${\mathbf w}_1,\ldots, {\mathbf w}_N$. 
Let $\boldsymbol\gamma=(\gamma_1,\ldots, \gamma_n)\in \mathbb C^n$. 
In \cite{GKZ1}, Gelfand, Kapranov and Zelevinsky define
the \emph{$A$-hypergeometric system} to be the system of
differential equations
\begin{eqnarray}\label{GKZeqn1}
\sum_{j=1}^N w_{ij} x_j\frac{\partial \Phi}{\partial x_j}+\gamma_i
\Phi=0 \quad (i=1,\ldots, n),
\end{eqnarray}
\begin{eqnarray}\label{GKZeqn2}
\square_{\boldsymbol\lambda}\Phi:=\Big(\prod_{\lambda_j>0} \Big(\frac{\partial}{\partial x_j}\Big)^{\lambda_j}-
\prod_{\lambda_j<0} \Big(\frac{\partial}{\partial x_j}\Big)^{-\lambda_j}\Big)\Phi=0,
\end{eqnarray}
where for the second system of equations, $\boldsymbol\lambda=(\lambda_1,\ldots,
\lambda_N)\in\mathbb Z^N$ goes over the family of integral linear
relations
$$\sum_{j=1}^N
\lambda_j{\mathbf w}_j=0$$ among ${\mathbf w}_1,\ldots, {\mathbf w}_N$. We call the corresponding 
$\mathcal D_{\mathbb A^N}$-module 
$$\mathcal H_{\boldsymbol\gamma}:=
\mathcal D_{\mathbb A^N}\Big/\mathcal D_{\mathbb A^N}\Big(\sum_{j=1}^N w_{ij} x_j\partial_{x_j}+\gamma_i,  \;
\square_{\boldsymbol\lambda}\Big)$$
the \emph{$A$-hypergeometric $\mathcal D$-module}. 
A solution of the $A$-hypergeometric system is given by
\begin{eqnarray}\label{eqn:gkzint}
f(x_1,\ldots, x_N)=\int_{\sigma} t_1^{\gamma_1}\cdots
t_n^{\gamma_n} e^{\sum_{j=1}^N x_jt_1^{w_{1j}}\cdots
t_n^{w_{nj}}}\frac{dt_1}{t_1}\cdots \frac{dt_n}{t_n}
\end{eqnarray}
where $\sigma$ is a real $n$-dimensional cycle in $\mathbb G_m^n$. (Confer 
\cite[(2.6)]{A}.) 

Adolphson (\cite[Theorem 3.9]{A}) proves that the $A$-hypergeometric $\mathcal D$-module is holonomic. 
Define the \emph{Newton polytope} $\Delta_\infty$ at $\infty$ to be the convex hull of
$\{0,\mathbf w_1,\ldots, \mathbf w_n\}$ in $\mathbb R^n$. 
If $\boldsymbol\gamma$ satisfies the so-called semi-nonresonance condition, 
Adolphson (\cite[Corollary 5.20]{A}) proves that the $A$-hypergeometric $\mathcal D$-module has rank 
$n!\mathrm{vol}(\Delta_\infty)/[\mathbb Z^n:M']$, where
$M'$ is the subgroup of $\mathbb Z^n$ generated by $\mathbf w_1,\ldots, \mathbf w_N$, and $\mathrm{vol}$ is the standard
Lebesgue measure. Without the semi-nonresonance condition, the $A$-hypergeometric $\mathcal D$-module may not have
the expected rank. The dependence of the rank on $\boldsymbol\gamma$ is thoroughly studied by
Matusevich-Miller-Walther \cite{MMW}. In this paper, we use 
cohomology functors on the derived category of $\mathcal D$-modules to define a stable GKZ hypergeometric $\mathcal D$-module 
so that it has the expected rank without any assumption on $\boldsymbol\gamma$.

\subsection{$\mathcal D$-modules associated to exponential functions and power functions}

Consider the $\mathcal D_{\mathbb A^1}$-module $\mathcal L=\mathcal O_{\mathbb A^1}e^{x}$. 
For section $f$ of $\mathcal O_{\mathbb A^1}$, we have 
$$\partial_x(f e^{x})=\Big(\frac{\partial f}{\partial x}+f\Big)e^{x}.$$ We have 
an isomorphism of $\mathcal D_{\mathbb A^1}$-modules 
$$\mathcal O_{\mathbb A^1}\stackrel\cong \to \mathcal L, \quad f\mapsto fe^{x}$$
so that the action of $\partial_x$ on a section $f$ of $\mathcal O_{\mathbb A^1}$ is given by 
$$\Big(e^{-x} \circ \frac{\partial}{\partial x} \circ e^{x}\Big)f=\frac{\partial f}{\partial x}+f.$$
$\mathcal L$ is an integrable connection on $\mathbb A^1$. 

Consider the $\mathcal D_{\mathbb G_m^n}$-module 
$$K_{\boldsymbol \gamma}=\mathcal O_{\mathbb G_m^n}\mathbf t^{\boldsymbol\gamma},$$
where $\mathbf t^{\boldsymbol\gamma}=t_1^{\gamma_1}\cdots t_n^{\gamma_n}$. 
For any section $g$ of $\mathcal O_{\mathbb G_m^n}$, we have 
$$\partial_{t_i}(g \mathbf t^{\boldsymbol\gamma})=\Big(\frac{\partial g}{\partial t_i}
+\frac{\gamma_i}{t_i}g\Big)\mathbf t^{\boldsymbol\gamma}.$$ 
We have an isomorphism of $\mathcal D_{\mathbb G_m^n}$-modules 
$$\mathcal O_{\mathbb G_m^n}\stackrel\cong\to \mathcal K_{\boldsymbol\gamma}, 
\quad g\mapsto g \mathbf t^{\boldsymbol\gamma}$$
so that the action of $\partial_{t_i}$ on a section $g$ of $\mathcal O_{\mathbb G_m^n}$ is given by 
$$\Big( \mathbf t^{-\boldsymbol\gamma} \circ \frac{\partial}{\partial t_i}  \circ \mathbf t^{\boldsymbol\gamma}\Big)g
=\frac{\partial g}{\partial t_i}+\frac{\gamma_i}{t_i} g.$$
$\mathcal K_{\boldsymbol\gamma}$ is an integrable 
connection on $\mathbb G^n_m$. 

\begin{remark} $\mathcal K_{\boldsymbol\gamma}$ depends only on $\boldsymbol\gamma\mod \mathbb Z^n$. Indeed, if 
$\boldsymbol \gamma=\boldsymbol \gamma'+\mathbf k$ for some $\mathbf k=(k_1,\ldots, k_n)\in\mathbb Z^n$, then we have an
isomorphism
$$\mathcal O_{\mathbb G_m^n}\stackrel\cong\to \mathcal O_{\mathbb G_m^n},\quad  g\mapsto \mathbf t^{\mathbf k} g$$
transforming $\mathbf t^{-\boldsymbol\gamma} \circ \frac{\partial}{\partial t_i}  \circ \mathbf t^{\boldsymbol\gamma}$
to $\mathbf t^{-\boldsymbol\gamma'} \circ \frac{\partial}{\partial t_i}  \circ \mathbf t^{\boldsymbol\gamma'}$.
For this reason, from now on, we denote $\mathcal K_{\boldsymbol\gamma}$ by $\mathcal K_{\overline{\boldsymbol\gamma}}$, where
$$\overline{\boldsymbol\gamma}=\boldsymbol\gamma+\mathbb Z^n\in \mathbb C^n/\mathbb Z^n.$$
\end{remark}

Let $$\pi_1: \mathbb G^n_m\times \mathbb A^N \to \mathbb G^n_m,\quad \pi_2: \mathbb G^n_m 
\times \mathbb A^N\to \mathbb A^N$$ be the projections, and let $F$
be the morphism
$$F:\mathbb G^n_m \times \mathbb A^N\to \mathbb A^1, \quad F(\mathbf t, \mathbf x)=\sum_{j=1}^N x_j \mathbf t^{\mathbf w_j},$$
where $\mathbf t=(t_1,\ldots, t_n)$, $\mathbf x=(x_1, \ldots, x_N)$ and $\mathbf t^{\mathbf w_j}=t_1^{w_{1j}}\cdots t_n^{w_{nj}}.$
The $\mathcal D_{\mathbb G^n_m\times \mathbb A^N}$-module 
$\pi_1^*\mathcal K_{\overline{\boldsymbol\gamma}}\otimes_{\mathcal O_{\mathbb G_m^n\times\mathbb A^N}}F^*\mathcal L$
is $\mathcal O_{\mathbb G^n_m\times \mathbb A^N}$ so that for section $h$ of $\mathcal O_{\mathbb G^n_m\times \mathbb A^N}$,  
we have 
\begin{align*}
\partial_{t_i}h&=\Big(\mathbf t^{-\boldsymbol\gamma}e^{-F(\mathbf t, \mathbf x)}\circ  \frac{\partial}{\partial t_i}
\circ \mathbf t^{\boldsymbol\gamma}e^{F(\mathbf t, \mathbf x)}\Big)h=
\frac{\partial h}{\partial t_i}+\frac{1}{t_i}\Big(\gamma_i +\sum_{j=1}^N x_jw_{ij} \mathbf t^{\mathbf w_j}\Big) h,\\
\partial_{x_j}h&=\Big(\mathbf t^{-\boldsymbol\gamma}e^{-F(\mathbf t, \mathbf x)}\circ \frac{\partial}{\partial x_j}\circ 
\mathbf t^{\boldsymbol\gamma}e^{F(\mathbf t, \mathbf x)}\Big)h=
\frac{\partial h}{\partial x_j}+ \mathbf t^{\mathbf w_j} h.
\end{align*}

\begin{definition} The \emph{stable GKZ hypergeometric $\mathcal D_{\mathbb A^N}$-module} is defined to be
$$\mathrm{Hyp}_{\overline{\boldsymbol\gamma}}= R\pi_{2+}(\pi_1^*\mathcal K_{\overline{\boldsymbol\gamma}}\otimes
_{\mathcal O_{\mathbb G^n_m\times\mathbb A^N}} F^* \mathcal L).$$ 
\end{definition}

Since $\pi_1^*\mathcal K_{\overline{\boldsymbol\gamma}}\otimes
_{\mathcal O_{\mathbb G^n_m\times\mathbb A^N}} F^* \mathcal L$ is an integrable connection, it is holonomic. By 
\cite[Theorem 3.2.3]{Hotta}, 
$\mathrm{Hyp}_{\overline{\boldsymbol\gamma}}$ is an object 
in the derived category $D_h^b(\mathcal D_{\mathbb A^N})$ of bounded complexes of $\mathcal D_{\mathbb A^N}$-modules
with holonomic cohomologies. In section 1, we show that $\mathrm{Hyp}_{\overline{\boldsymbol\gamma}}$ is actually 
a holonomic $\mathcal D_{\mathbb A^N}$-module, that is, $R^q\pi_{2+}(\pi_1^*\mathcal K_{\overline{\boldsymbol\gamma}}\otimes
_{\mathcal O_{\mathbb G^n_m\times\mathbb A^N}} F^* \mathcal L)$ for $q\not=0$.

A complex Laurent polynomial $$f=\sum_{j=1}^M a_j\mathbf t^{\mathbf w_j}$$
is called \emph{nondegenerate with respect to $\Delta_\infty$} if 
for any face $\Gamma$ of $\Delta_\infty$ not containing the origin, the system of equations 
$$\frac{\partial f_\Gamma}{\partial t_1}=\cdots =\frac{\partial f_\Gamma}{\partial t_n}=0$$ has no solution 
in $(\mathbb C-\{0\})^n$, where $f_\Gamma=\sum_{\mathbf w_j\in \Gamma}a_j \mathbf t^{\mathbf w_j}$. 
The main result of this paper is the following theorem.

\begin{theorem}\label{thm:holonomic} ${}$
\begin{enumerate}[(i)]
\item $\mathrm{Hyp}_{\overline{\boldsymbol\gamma}}$ 
is a holonomic $\mathcal D_{\mathbb A^N}$-module. 
\item 
Let $U$ be a Zariski open subset of $\mathbb A^N$ so that for any rational point 
$\mathbf a=(a_1, \ldots,a_N)$ in $U$,  the Laurent polynomial
$F({\mathbf t}, \mathbf a)=\sum_{j=1}^N a_j \mathbf  t^{\mathbf w_j}$ 
is nondegenerate with respect to $\Delta_\infty$. Then 
the restriction to $U$ of $\mathrm{Hyp}_{\overline{\boldsymbol\gamma}}$
is an integrable connection of rank 
$n!\mathrm{vol}(\Delta_\infty)$. 
\end{enumerate}
\end{theorem}

\subsection{de Rham complex} By definition, we have
$$\mathrm{Hyp}_{\overline{\boldsymbol\gamma}}\cong R\pi_{2*}(\mathcal D_{\mathbb A^N\leftarrow 
\mathbb G_m^n\times\mathbb A^N}
\otimes^L_{\mathcal D_{\mathbb G_m^n\times\mathbb A^N}}(\pi_1^*\mathcal K_{\overline{\boldsymbol\gamma}}\otimes
_{\mathcal O_{\mathbb G_m^n\times\mathbb A^N}}F^*\mathcal L)).$$
By \cite[Lemma 1.5.27]{Hotta} and the discussion after it, in the derived category of 
$(\pi_2^{-1}\mathcal D_{\mathbb A^N}, \mathcal D_{\mathbb G_m^n\times\mathbb A^N})$-modules, we have 
$$\mathcal D_{\mathbb A^N\leftarrow \mathbb G_m^n\times\mathbb A^N}
\cong \omega_{\mathbb G_m^n}\boxtimes \mathcal D_{\mathbb A^N}
\cong(\boldsymbol\Omega^\cdot_{\mathbb G_m^n}[n]\otimes_{\mathcal O_{\mathbb G_m^n}}\mathcal D_{\mathbb G_m^n})\boxtimes 
\mathcal D_{\mathbb A^N},$$ 
where righthand side is the relative de Rham complex, and $\omega_{\mathbb G^n_m}$ is the right 
$\mathcal D_{\mathbb G^n_m}$-module of top differential forms.
Since $\pi_2$ is affine, we have $R\pi_{2*}\cong \pi_{2*}.$ We thus have 
\begin{align*}
\mathrm{Hyp}_{\overline{\boldsymbol\gamma}}&\cong \pi_{2*}\Big(\big(
(\boldsymbol\Omega^\cdot_{\mathbb G_m^n}[n]\otimes_{\mathcal O_{\mathbb G_m^n}}\mathcal D_{\mathbb G_m^n})\boxtimes 
\mathcal D_{\mathbb A^N}\big)\otimes_{\mathcal D_{\mathbb G_m^n\times\mathbb A^N}}
(\pi_1^*\mathcal K_{\overline{\boldsymbol\gamma}}\otimes
_{\mathcal O_{\mathbb G_m^n\times\mathbb A^N}}F^*\mathcal L)\Big)\\
&\cong \pi_{2*}(\pi_1^* \boldsymbol\Omega^\cdot_{\mathbb G_m^n}\otimes_{\mathcal O_{\mathbb G_m^n\times \mathbb A^N}} 
\pi_1^*\mathcal K_{\overline{\boldsymbol\gamma}}\otimes_{\mathcal O_{\mathbb G_m^n\times \mathbb A^N}} F^*\mathcal L)[n].
\end{align*}
where the last expression is the relative de Rham complex for the integrable connection 
$\pi_1^*\mathcal K_{\overline{\boldsymbol\gamma}}\otimes_{\mathcal O_{\mathbb G_m^n\times \mathbb A^N}} F^*\mathcal L$.
Let 
$\mathcal D_{\overline{\boldsymbol\gamma}}^\cdot= \pi_{2*}(\pi_1^* \boldsymbol\Omega^\cdot_{\mathbb G_m^n}
\otimes_{\mathcal O_{\mathbb G_m^n\times \mathbb A^N}} 
\pi_1^*\mathcal K_{\overline{\boldsymbol\gamma}}\otimes_{\mathcal O_{\mathbb G_m^n\times \mathbb A^N}} F^*\mathcal L).$
We have 
$$\mathcal D_{\overline{\boldsymbol\gamma}}^q \cong \bigoplus_{1\leq i_1<\cdots<i_q\leq n} 
\pi_{2*}\mathcal O_{\mathbb G_m^n\times \mathbb A^N}
\frac{\mathrm dt_{i_1}}{t_{i_1}}\wedge\cdots\wedge \frac{\mathrm dt_{i_q}}{t_{i_q}}.$$
and the operator $d^q: \mathcal D_{\overline{\boldsymbol\gamma}}^q\to \mathcal D_{\overline{\boldsymbol\gamma}}^{q+1}$ is given by 
\begin{align}\label{aln:d}
&\;d\Big(h\frac{\mathrm dt_{i_1}}{t_{i_1}}\wedge\cdots\wedge 
\frac{\mathrm dt_{i_q}}{t_{i_q}}\Big)= \sum_{i=1}^n\Big(t_i\frac{\partial h}{\partial t_i}+\gamma_i h+
\sum_j w_{ij} x_j\mathbf t^{\mathbf w_j} h\Big)\frac{\mathrm dt_i}{t_i}\wedge\frac{\mathrm dt_{i_1}}{t_{i_1}}\wedge\cdots\wedge 
\frac{\mathrm dt_{i_q}}{t_{i_q}} \\
&\qquad\qquad=\Big(\mathbf t^{-\boldsymbol \gamma}e^{-F(\mathbf t, \mathbf x)}\circ 
\mathrm d_{\mathbf t} \circ \mathbf t^{\boldsymbol \gamma}e^{F(\mathbf t, \mathbf x)}\Big)
\Big(h\frac{\mathrm dt_{i_1}}{t_{i_1}}\wedge\cdots\wedge 
\frac{\mathrm dt_{i_q}}{t_{i_q}}\Big)\nonumber,
\end{align}
where $\mathrm d_{\mathbf t}$ is the exterior derivative relative to the variable $\mathbf t$. 
This is a complex of $\mathcal D_{\mathbb A^N}$-modules so that $\partial_{x_j}$ acts via
\begin{align}\label{aln:nabla}
\nabla_{\partial_{x_j}}=\mathbf t^{-\boldsymbol\gamma}e^{-F(\mathbf t,\mathbf x)}\circ \frac{\partial}{\partial x_j}
\circ \mathbf t^{\boldsymbol\gamma}e^{F(\mathbf t,\mathbf x)}.
\end{align}
We thus get the following description of $\mathrm{Hyp}_{\overline{\boldsymbol\gamma}}$ by the de Rham complex.

\begin{proposition}\label{prop:bigdR} Notation as above. We have 
$\mathrm{Hyp}_{\overline{\boldsymbol\gamma}}\cong \mathcal D_{\overline{\boldsymbol\gamma}}^\cdot[n]$.
\end{proposition}

To prove Theorem \ref{thm:holonomic}, we need to work with another de Rham complex which we now describe. Let
$$\delta=\{c_1\mathbf w_1+\cdots+c_n\mathbf w_N: c_1,\ldots, c_N\in \mathbb R_{\geq 0}\}$$
be the convex polyhedral cone generated by $\mathbf w_1,\ldots,\mathbf w_N$, 
let $\mathbb C[\delta\cap \mathbb Z^n]$
be the subring of $\mathbb C[\mathbf t^{\pm 1}]$
generated by monomials $\mathbf t^{\mathbf w}$ ($\mathbf w\in \delta\cap \mathbb Z^n$), 
let $q_2: \mathrm{Spec}\, \mathbb C[\delta\cap \mathbb Z^n]\times \mathbb A^N\to
\mathbb A^N$ be projection, let $(\mathcal C_{\boldsymbol\gamma}^\cdot, d)$ be the complex so that 
\begin{align*}
&\mathcal C_{\boldsymbol\gamma}^q=\bigoplus_{1\leq i_1<\cdots< i_q\leq n} q_{2*}
\mathcal O_{\mathrm{Spec}\, \mathbb C[\delta\cap \mathbb Z^n]\times\mathbb A^N} \frac{\mathrm dt_{i_1}}{t_{i_1}}\wedge\cdots\wedge 
\frac{\mathrm dt_{i_q}}{t_{i_q}},
\end{align*}
$d_q: \mathcal C_{\boldsymbol\gamma}^q\to \mathcal C_{\boldsymbol\gamma}^{q+1}$ is given 
by the formula (\ref{aln:d}), and $\partial_{x_j}$ is given by the formula
(\ref{aln:nabla}). In section 2, we prove the following.

\begin{proposition}\label{prop:smalldR} There exists a representative $\boldsymbol\gamma\in\mathbb C^n$ of 
$\overline{\boldsymbol\gamma}\in\mathbb C^n/\mathbb Z^n$ such that 
$\mathrm{Hyp}_{\overline{\boldsymbol\gamma}}\cong \mathcal C^\cdot_{\boldsymbol\gamma}[n].$
\end{proposition} 
 
In section 3, we prove the following. 

\begin{proposition}\label{prop:goodR} 
Let $U$ be a Zariski open subset of $\mathbb A^N$ so that for any rational point 
$\mathbf a=(a_1, \ldots,a_N)$ in $U$,  the Laurent polynomial
$F({\mathbf t}, \mathbf a)=\sum_{j=1}^N a_j \mathbf  t^{\mathbf w_j}$ 
is nondegenerate with respect to $\Delta_\infty$. Then $\mathcal H^i(\mathcal C^\cdot_{\boldsymbol\gamma})|_U=0$
for $i\not=n$, and $\mathcal H^n(\mathcal C^\cdot_{\boldsymbol\gamma})|_U$
is an integrable connection of rank 
$n!\mathrm{vol}(\Delta_\infty)$. 
\end{proposition}

Note that Theorem \ref{thm:holonomic} (ii) follows from Propositions \ref{prop:smalldR} and \ref{prop:goodR}. 

\subsection{The better behaved GKZ system}
In \cite[Definition 2.1]{BP}, Borisov and Horja introduced a better behaved GKZ system. It is 
the following system of partial differential equations 
for family of function $\{\Phi_{\mathbf v}(x_1,\ldots, x_N): \mathbf v=(v_1,\ldots, v_n)\in \delta\cap \mathbb Z^n\}$:
\begin{eqnarray}\label{betterGKZeqn1}
\frac{\partial \Phi_{\mathbf v}}{\partial x_j}=\Phi_{\mathbf v+\mathbf w_j},
\end{eqnarray}
\begin{eqnarray}\label{betterGKZeqn2}
\sum_{j=1}^N w_{ij} x_j\frac{\partial \Phi_{\mathbf v}}{\partial x_j}+(\gamma_i+v_i) \Phi_{\mathbf v}=0 \quad (i=1,\ldots, n).
\end{eqnarray}
Suppose $\{\Phi_{\mathbf v}\}_{\mathbf v\in\delta\cap\mathbb Z^n}$ 
is a solution of the better behaved GKZ system. We claim that each $\Phi_{\mathbf v}$ 
is a solution of the $A$-hypergeometric $\mathcal D$-module $\mathcal H_{\boldsymbol \gamma+\mathbf v}$.
Indeed, (\ref{betterGKZeqn2}) is the equation of (\ref{GKZeqn1}) with the parameter $\boldsymbol\gamma$ replaced by 
$\boldsymbol \gamma+\mathbf v$. Suppose we have an integral linear relation
$$\lambda_1{\mathbf w}_1+\cdots + \lambda_N\mathbf w_N=0.$$ 
Then we have 
\begin{eqnarray*}
\Big(\prod_{\lambda_j>0} \Big(\frac{\partial}{\partial x_j}\Big)^{\lambda_j}\Big)\Phi_{\mathbf v}
&=&\Phi_{\mathbf v+\sum_{\lambda_j>0} \lambda_j \mathbf w_j},\\
\Big(\prod_{\lambda_j<0} \left(\frac{\partial}{\partial x_j}\right)^{-\lambda_j}\Big)\Phi_{\mathbf v}
&=&\Phi_{\mathbf v-\sum_{\lambda_j< 0} \lambda_j \mathbf w_j}.
\end{eqnarray*}
Hence $\Phi_{\mathbf v}$ also satisfies (\ref{GKZeqn2}). This proves our claim.
If $\Phi$ is a solution of $\mathcal H_{\boldsymbol\gamma+\mathbf v}$, 
then $\frac{\partial\Phi}{\partial x_k}$ is a solution of $\mathcal H_{\boldsymbol\gamma+\mathbf v+\mathbf w_k}$. 
This follows from the equation
$$\Big(\sum_{j=1}^n w_{ij} x_j\partial_{x_j}+\gamma_i+v_i+w_{ik}\Big)\partial_{x_k}
=\partial_{x_k} \Big(\sum_{j=1}^n w_{ij} x_j\partial_{x_j}+\gamma_i+v_i\Big).$$
We can construct a direct system
$$(\mathcal H_{\boldsymbol\gamma+\mathbf v}, \partial_{x_1}, \ldots, \partial_{x_N})_{\mathbf v\in \delta\cap\mathbb Z^n}$$
such that the index set is $\delta\cap\mathbb Z^n$ with the partial order 
$$\mathbf v^{(1)}\leq \mathbf v^{(2)}\Leftrightarrow \mathbf v^{(1)}-\mathbf v^{(2)}=a_1\mathbf w_1+\cdots+
a_N\mathbf w_N\hbox{ for some nonnegative integers } a_j,$$
and the transition morphisms are composites of 
$$\cdot\partial_{x_k}: \mathcal D_{\mathbb A^N}\Big/\mathcal D_{\mathbb A^N}
\Big(\sum_{j=1}^N w_{ij} x_j\partial_{x_j}+\gamma_i+v_i+w_{ik},  \square_{\boldsymbol\lambda}\Big)\to 
\mathcal D_{\mathbb A^N}\Big/\mathcal D_{\mathbb A^N}
\Big(\sum_{j=1}^N w_{ij} x_j\partial_{x_j}+\gamma_i+v_i,  \square_{\boldsymbol \lambda}\Big)$$ 
induced by the right multiplication by $\partial_{x_k}$: 
$$\cdot \partial_{x_k}: \mathcal D_{\mathbb A^N}\to \mathcal D_{\mathbb A^N},\quad P\mapsto P\partial_{x_k}.$$
The $\mathcal D$-module underlying the better behaved GKZ system is 
$$\varinjlim_{\mathbf v\in  \delta\cap\mathbb Z^n} \mathcal H_{\boldsymbol\gamma+\mathbf v}.$$
Solutions of this $\mathcal D$-module are the solutions of the better behaved GKZ system.
One needs to be careful with the direct limit since $(\delta\cap\mathbb Z^n, \leq)$ is not a direct set. 
We prove the following in Section 2.

\begin{proposition}\label{prop:betterGKZ} We have an isomorphism
$\varinjlim_{\mathbf v\in  \delta\cap\mathbb Z^n} \mathcal H_{\boldsymbol\gamma+\mathbf v}\cong 
\mathcal H^n(\mathcal C_{\boldsymbol\gamma}).$
\end{proposition}

\begin{remark}
The stable GKZ hypergeometric $\mathcal D$-module already appears implicitly in \cite[Section 4]{GKZ2}.
Suppose $\mathbf w_1,\ldots, \mathbf w_N$ generates $\mathbb Z^n$
and $0$ is the only unit in the semigroup generated by $\mathbf w_1,\ldots, \mathbf w_N$. 
Schulze and Walther \cite[Corollary 3.8]{SW} prove that the $A$-hypergeometric $\mathcal D$-module 
$\mathcal H_{\boldsymbol\gamma}$ is isomorphic to the stable GKZ hypergeometric 
$\mathcal D$-module $\mathrm{Hyp}_{\overline{\boldsymbol\gamma}}$ if and only if 
$\boldsymbol\gamma$ is not a strongly resonant parameter (\cite[Definition 3.5]{SW}).
For any $\overline{\boldsymbol\gamma}$, we can always choose a representative 
$\boldsymbol\gamma$ of $\overline{\boldsymbol\gamma}$ which is not strongly resonant.
Thus the stable GKZ hypergeometric 
$\mathcal D$-module $\mathrm{Hyp}_{\overline{\boldsymbol\gamma}}$
is isomorphic to an $A$-hypergeometric $\mathcal D$-module 
for a properly chosen $\boldsymbol\gamma$. 
Moreover, by Propositions \ref{prop:smalldR} and \ref{prop:betterGKZ}, we can write the stable GKZ hypergeometric 
$\mathcal D$-module $\mathrm{Hyp}_{\overline{\boldsymbol\gamma}}$ as the direct limit 
of the $A$-hypergeometric $\mathcal D$-modules $\mathcal H_{\boldsymbol\gamma+\mathbf v}$ ($\mathbf v\in
\delta\cap\mathbb Z^n$). The idea of looking at families of $A$-hypergeometric $\mathcal D$-modules 
with respect to the contiguity operators $\partial_{x_j}$ first appears in \cite[Remark 3.8]{W}. 
For this reason, we call the $\mathcal D$-module $\mathrm{Hyp}_{\overline{\boldsymbol\gamma}}$
the stable  GKZ hypergeometric $\mathcal D$-module.
If  $\mathbf w_1,\ldots, \mathbf w_N$ generates $\mathbb Z^n$
and $0$ is the only unit in the semigroup generated by $\mathbf w_1,\ldots, \mathbf w_N$, 
we can deduce Theorem \ref{thm:holonomic} from \cite[Corollary 3.8]{SW} and 
the known properties of the $A$-hypergeometric $\mathcal D$-module.
In this paper, we give a direct proof of Theorem \ref{thm:holonomic} in general
using the six-functor formalism in the theory of algebraic $\mathcal D$-modules, and the classical methods of 
Deligne \cite{D} and of Kouchnirenko \cite{K}. 
\end{remark}

\begin{remark}\label{rm:hypmotive} Unlike the $A$-hypergeometric $\mathcal D$-module, the stable GKZ hypergeometric 
$\mathcal D$-module $\mathrm{Hyp}_{\overline{\boldsymbol\gamma}}$
has the expected rank $n!\mathrm{vol}(\Delta_\infty)$ for all $\boldsymbol\gamma$. 
Our study of the stable GKZ hypergeometric $\mathcal D$-module is motivated by that of 
the $\ell$-adic GKZ hypergeometric sheaf introduced in \cite{F}. 
Let $p$ and $\ell$ be two distinct prime numbers, 
$$\pi_1: \mathbb G^n_{m,\mathbb F_p}\times_{\mathbb F_p}\mathbb A^N_{\mathbb F_p} 
\to \mathbb G^n_{m,\mathbb F_p},\quad \pi_2: \mathbb G^n_{m,\mathbb F_p}\times_{\mathbb F_p} \mathbb A_{\mathbb F_p}^N 
\to \mathbb A_{\mathbb F_p}^N, \quad F:\mathbb G^n_{m,\mathbb F_p} \times_{\mathbf F_p} \mathbb A^N_{\mathbb F_p}\to
\mathbb A_{\mathbb F_p}^1$$
the projections and the morphism defined by the Laurent polynomial $F$ 
over the finite ground field $\mathbb F_p=\mathbb Z/p$, 
$\psi: \mathbb F_p\to \overline{\mathbb Q}^*_\ell$
the additive character $\psi(t)=e^{\frac{2\pi i t}{p}}$,  $\chi_1, \ldots, \chi_n:\mathbb F_p^*\to\overline{\mathbb Q}^*_\ell$
multiplicative characters, $\boldsymbol\chi$ the character
$$\boldsymbol\chi: (\mathbb F_p^*)^n\to \overline{\mathbb Q}^*_\ell, \quad (t_1, \ldots, t_n)\mapsto \chi_1(t_1)\cdots\chi_n(t_n),$$ and
$\mathcal K_{\boldsymbol\chi}$ (resp. $\mathcal L_\psi$) the Kummer sheaf (resp. Artin-Schreier sheaf) on 
$\mathbb G^n_{m,\mathbb F_p}$ (resp. $\mathbb A^1_{\mathbb F_p}$). In \cite{F}, 
we define the $\ell$-adic GKZ hypergeometric sheaf to be 
$$\mathrm{Hyp}_{\boldsymbol\chi}=R\pi_{2!} (\pi_1^*\mathcal K_{\boldsymbol\chi}\otimes F^*\mathcal L_\psi)[n+N].$$
By the Grothendieck trace formula, for any rational point $\mathbf x=(x_1,\ldots, x_N)$ of $\mathbb A^N_{\mathbb F_p}$, we have 
\begin{eqnarray}\label{eqn:twistexpsum}
\mathrm{Tr}(\mathrm{Frob}_{\mathbf x}, \mathrm{Hyp}_{\boldsymbol\chi,\overline{\mathbf x}})
=(-1)^{n+N}\sum_{t_1,\ldots, t_n\in\mathbb F_p^*}\chi_1(t_1)\cdots\chi_n(t_n)e^{\frac{2\pi i}{p}\sum_{j=1}^N x_j t_1^{w_{1j}}\cdots
t_n^{w_{nj}}},
\end{eqnarray}
where $\mathrm{Frob}_{\mathbf x}$ is the Frobenius element at $\mathbf x$. The twisted exponential sum on the 
righthand side of (\ref{eqn:twistexpsum}) is an arithmetic analogue of the integral (\ref{eqn:gkzint}). It is the hypergeometric 
function over the finite field introduced by Gelfand and Graev  \cite{GG}. In \cite{F}, we prove the $\ell$-adic GKZ hypergeometric sheaf
$\mathrm{Hyp}_{\boldsymbol\chi}$ is a perverse sheaf on $\mathbb A^N$, and it comes from a lisse $\ell$-adic sheaf 
of rank $n!\mathrm{vol}(\Delta_\infty)$ when restricted 
to the open subset parametrizing nondegenerate Laurent polynomials. In \cite{FZ}, we introduce the $p$-adic GKZ hypergeometric 
complex. In the derived category of arithmetic $\mathcal D^\dagger$-modules, it is exactly
$R\pi_{2+}(\pi_1^*\mathcal K_{\boldsymbol\gamma}\otimes F^* \mathcal L_\psi)$,
where $\mathcal L_\psi$ is the Dwork isocrystal  associated to $\psi$ and $\mathcal K_{\boldsymbol\gamma}$ is the Kummer
isocrystal associated to an $n$-tuple $\boldsymbol\gamma$ of rational numbers. We prove
it is overholonomic and comes from an $F$-isocrystal of rank 
$n!\mathrm{vol}(\Delta_\infty)$ when restricted 
to the open subset parametrizing nondegenerate Laurent polynomials. We expect that there exists a category of exponential motives 
and a GKZ hypergeometric motive in this category so that the stable GKZ hypergeometric $\mathcal D$-module, the $\ell$-adic GKZ
hypergeometric sheaf, and the $p$-adic GKZ hypergeometric complex are the de Rham, the $\ell$-adic, and the $p$-adic realizations 
of the GKZ hypergeometric motive, respectively.
\end{remark}

The paper is organized as follows. In Section 1, we describe the stable GKZ hypergeometric $\mathcal D$-module 
$\mathrm{Hyp}_{\overline{\boldsymbol\gamma}}$  in terms of the
Fourier transformation and prove it is holonomic. In Section 2, we prove  
Proposition \ref{prop:smalldR} expressing $\mathrm{Hyp}_{\overline{\boldsymbol\gamma}}$
by the de Rham complex $C_{\boldsymbol\gamma}^\cdot$. 
We also Proposition \ref{prop:betterGKZ} which relates
$C_{\boldsymbol\gamma}^\cdot$ with the better behaved GKZ system of Borisov-Horja.
In Section 3, we prove Proposition \ref{prop:goodR}. 

\subsection{Acknowledgement} I thank Lev Borisov, Jiangxue Fang, Paul Horja, and Hao Zhang 
for inspiring discussions. I am especially grateful to Uli Walther for many suggestions and detailed explanations about the
GKZ hypergeometric system. This research is supported by 2023YFA1009703.

\section{Fourier transform and the stable GKZ hypergeometric $\mathcal D$-module}

Let $\mathcal M$ be a quasi-coherent left $\mathcal D_{\mathbb A^N}$-module, 
let $\mathbb A^{*N}$ be the dual of $\mathbb A^N$, and let 
$$D=\Gamma(\mathbb A^N,\mathcal D_{\mathbb A^N}),
\quad D'=\Gamma(\mathbb A^{*N},\mathcal D_{\mathbb A^{*N}}),\quad M=\Gamma(\mathbb A^N,\mathcal M).$$
Choose a linear coordinate $(x_1, \ldots, x_N)$ on $\mathbb A^N$, and let  
$(\xi_1, \ldots, \xi_N)$ be the dual coordinate on $\mathbb A^{*N}$. 
The \emph{Fourier transform} $\widehat M$ of $M$ is defined by the following conditions:
\begin{enumerate}[(1)]
\item $\widehat M$ is a left $D'$-module.
\item As a vector space, $\widehat M$ coincides with $M$. 
\item For any $s\in M$, let $\hat s$
be the element in $\widehat M$ corresponding to $s$. We have
$$\xi_j \hat s= -\widehat{(\partial_{x_j}s)},\quad \partial_{\xi_j}\hat s=\widehat{(x_j s)}.$$
\end{enumerate}
The \emph{Fourier transform} $\widehat {\mathcal M}$ of $\mathcal M$ is 
the left $\mathcal D_{\mathbb A^{*N}}$-module so that 
$$\Gamma(\mathbb A^{*N},\widehat {\mathcal M})\cong \widehat M.$$
If $\mathcal M$ is a coherent (resp. holonomic) $\mathcal D_{\mathbb A^N}$-module, then 
$\widehat{\mathcal M}$ is coherent (resp. holonomic). The Fourier transform is an exact functor and it can be extended to a functor
$$\mathfrak F: D_{qc}^b(\mathcal D_{\mathbb A^N})\to D_{qc}^b(\mathcal D_{\mathbb A^{*N}})$$
on the derived categories of bounded complexes of $\mathcal D$-modules with quasi-coherent cohomologies. 

On a smooth variety $X$, recall that for any $K_1, K_2\in\mathrm{ob}\,D^b_{qc}(\mathcal D_X)$, we have
$$K_1\otimes_{\mathcal O_X}^L K_2=L\Delta^*(K_1\boxtimes K_2),$$
where $\Delta: X\to X\times X$ is the diagonal morphism. 
We define 
$$K_1\otimes_{\mathcal O_X}^! K_2=R\Delta^!(K_1\boxtimes K_2)\cong K_1\otimes_{\mathcal O_X}^L K_2[-\mathrm{dim}\,X].$$ 
Let $p_1: \mathbb A^N\times \mathbb A^{*N}\to  \mathbb A^N$ and 
$p_2: \mathbb A^N\times \mathbb A^{*N}\to  \mathbb A^{*N}$ be projections, and let 
$$\langle\, ,\,\rangle:  \mathbb A^N\times \mathbb A^{*N}\to \mathbb A^1, 
\quad ((x_1, \ldots, x_N), (\xi_1, \ldots, \xi_N))\mapsto x_1\xi_1+\cdots + x_N\xi_N$$ be the 
duality pairing. 
By \cite[7.1.4]{KL}, we have an isomorphism of functors 
$$\mathfrak F(-)\cong  Rp_{2+} (Rp_1^!(-)\otimes^!_{\mathcal O_{\mathbb A^N\times\mathbb A^{*N}}} 
R\langle\, ,\,\rangle^!\mathcal L)[1-N].$$

\begin{proposition}\label{prop:FourGKZ} Let $\iota$ be the morphism 
$$\iota: \mathbb G_m^n\to\mathbb A^N,\quad \mathbf t\mapsto (\mathbf t^{\mathbf w_1},\ldots, 
\mathbf t^{\mathbf w_N}).$$ Then $\iota$ is quasi-finite and affine, and we have 
$$\mathrm{Hyp}_{\overline{\boldsymbol\gamma}}=\mathfrak F(R\iota_+\mathcal K_{\overline{\boldsymbol\gamma}}).$$ 
Moreover, $\mathrm{Hyp}_{\overline{\boldsymbol\gamma}}$ is 
a holonomic $\mathcal D_{\mathbb A^N}$-module.
\end{proposition} 

\begin{proof} It is clear that $\iota$ is affine. To show it is quasi-finite, it suffices to show the morphism 
$$\iota_0: \mathbb G_m^n\to \mathbb G_m^N,\quad \mathbf t\mapsto (\mathbf t^{\mathbf w_1},\ldots, 
\mathbf t^{\mathbf w_N})$$ induced by $\iota$ is finite. Since the matrix $A=(\mathbf w_1,\ldots, \mathbf w_N)$
has rank $n$, we can choose coordinates on $\mathbb G_m^n$ and $\mathbb G_m^N$ so that the morphism
$\iota_0$ is given by $$(t_1,\ldots, t_n)\mapsto (t_1^{d_1}, \ldots, t_n^{d_n}, 1,\ldots, 1)$$ for a sequence of
integers $d_1|d_2|\cdots|d_n$, and hence is a finite morphism. 
Fix notation by the following commutative diagram:
$$\begin{tikzcd}
\mathbb G_m^n\times \mathbb A^N\arrow[r,"\iota\times\mathrm{id}"]\arrow[d,"\pi_1"]&
\mathbb A^N\times \mathbb A^N\arrow[r,"p_2"]\arrow[d,"p_1"]\arrow[dr,"{\langle\,,\,\rangle}"]&\mathbb A^N\\
\mathbb G_m^n\arrow[r,"\iota"]&\mathbb A^N&\mathbb A^1.
\end{tikzcd}$$
By the base change theorem \cite[Theorem 1.7.3]{Hotta}, the projection formula \cite[Corollary 1.7.5]{Hotta} and the fact that $\mathcal K_{\overline{\boldsymbol\gamma}}$ and 
$\mathcal L$ are integrable connections, we have 
\begin{eqnarray*}
\mathfrak F(R\iota_+\mathcal K_{\overline{\boldsymbol\gamma}})
&\cong& Rp_{2+} \Big(Rp_1^! R\iota_+\mathcal K_{\overline{\boldsymbol\gamma}}\otimes
_{\mathcal O_{\mathbb A^N\times\mathbb A^N}}^! R\langle\,,\,\rangle^! \mathcal L\Big)[1-N]\\
&\cong& Rp_{2+} \Big(R(\iota\times\mathrm{id})_+ R\pi_1^! \mathcal K_{\overline{\boldsymbol\gamma}}
\otimes_{\mathcal O_{\mathbb A^N\times\mathbb A^N}}^! R\langle\,,\,\rangle^! \mathcal L\Big)[1-N]\\
&\cong& Rp_{2+}\Big (R(\iota\times\mathrm{id})_+ \pi_1^* \mathcal K_{\overline{\boldsymbol\gamma}}\otimes
_{\mathcal O_{\mathbb A^N\times\mathbb A^N}} \langle\,,\,\rangle^*\mathcal L\Big)\\
&\cong& R(p_{2} \circ (\iota\times\mathrm{id}))_+ \Big(\pi_1^*\mathcal K_{\overline{\boldsymbol\gamma}}\otimes
_{\mathcal O_{\mathbb G_m^n\times\mathbb A^N}} 
(\langle\,,\,\rangle\circ (\iota\times \mathrm{id}))^* \mathcal L\Big)\\
&\cong&R\pi_{2+} (\pi_1^* \mathcal K_{\overline{\boldsymbol\gamma}}\otimes_{\mathcal O_{\mathbb G_m^n\times\mathbb A^N}} F^* \mathcal L)
\cong\mathrm{Hyp}_{\overline{\boldsymbol\gamma}}.
\end{eqnarray*}
Let $j: \mathbb G_m^N\hookrightarrow \mathbb A^N$ be the open immersion. By \cite[Theorem 3.2.3]{Hotta}, for any 
holonomic $\mathcal D_{\mathbb G^N_m}$-module $\mathcal N$, 
$Rj_+\mathcal N$ lies in the derived category $D^b_h(\mathcal D_{\mathbb A^N})$ of bounded complexes of 
$\mathcal D_{\mathbb A^N}$-modules with holonomic cohomologies. Since $j$ is an affine open immersion, we have
$Rj_+\mathcal N=j_*\mathcal N$. So $Rj_+\mathcal N$ is a holonomic $\mathcal D_{\mathbb A^N}$-module.
Since $\iota_0: \mathbb G_m^n\to \mathbb G_m^N$ is finite, we have 
$$R\iota_{0+}(-)=\iota_{0*}(\mathcal D_{\mathbb G^N_m\leftarrow \mathbb G^n_m}\otimes^L_{\mathcal D_{\mathbb G^n_m}} -).$$
It follows that $R\iota_{0+}$ maps $D^{\leq 0}_h(\mathcal D_{\mathbb G_m^n})$ to $D^{\leq 0}_h(\mathcal D_{\mathbb G_m^N})$. 
On the other hand, we have $$R\iota_{0+}\cong R\iota_{0!}\cong \mathbb D R\iota_{0+}\mathbb D,$$ 
where $\mathbb D$ is the duality functor.
So $R\iota_{0+}$ maps $D^{\geq 0}_h(\mathcal D_{\mathbb G_m^n})$ to $D^{\geq 0}_h(\mathcal D_{\mathbb G_m^N})$. 
Hence for any holonomic $\mathcal D_{\mathbb G^n_m}$-module $\mathcal M$, 
$R\iota_{0+}(\mathcal M)$ is a holonomic $\mathcal D_{\mathbb G^N_m}$-module. Thus 
$R\iota_+\mathcal K_{\overline{\boldsymbol\gamma}}=Rj_+ R\iota_{0+}\mathcal K_{\overline{\boldsymbol\gamma}}$ is 
a holonomic $\mathcal D_{\mathbb A^N}$-module. Therefore its Fourier transform $\mathrm{Hyp}_{\overline{\boldsymbol\gamma}}$ 
is a holonomic $\mathcal D_{\mathbb A^N}$-module. 
\end{proof}

\section{The de Rham complex for $\mathrm{Hyp}_{\overline{\boldsymbol\gamma}}$} 

In this section, we prove Propositions \ref{prop:smalldR} and \ref{prop:betterGKZ}.

Let $\mathbb Z^n$ be the standard lattice in $\mathbb R^n$, let $\Sigma$ be a regular fan of 
rational convex polyhedral cones in
the dual space $(\mathbb R^n)^*$ of $\mathbb R^n$, let $X(\Sigma)$ be the toric variety defined by $\Sigma$, and
let $$j: \mathbb G_m^n=\mathrm{Spec}\,\mathbb C[\mathbb Z^n]\hookrightarrow X(\Sigma)$$ be the open immersion 
of the open dense torus. Since $j$ is affine, we have 
$$\mathcal Rj_+\mathcal N\cong j_+\mathcal N$$ for any quasi-coherent $\mathcal D_{\mathbb G_m^n}$-module $\mathcal N$. 
The complement $D=X(\Sigma)-\mathbb G_m^n$ is a divisor with normal crossing in $X(\Sigma)$. 
Let $\boldsymbol\Omega^\cdot_{X(\Sigma)}(\log D)$ be the de Rham complex of differential forms with 
logarithmic poles along $D$. By \cite[4.3]{Fulton}, we have an isomorphism 
$$\mathcal O_{X(\Sigma)}\otimes_{\mathbb Z} \mathbb Z^n \stackrel\cong\to \boldsymbol\Omega^1_{X(\Sigma)}(\log D),\quad
f\otimes \mathbf w\mapsto f\frac{\mathrm d\mathbf t^{\mathbf w}}{\mathbf t^{\mathbf w}}$$
So $\boldsymbol\Omega^1_{X(\Sigma)}(\log D)$ is a globally free $\mathcal O_{X(\Sigma)}$-module with basis 
$\frac{\mathrm dt_1}{t_1}, \ldots, \frac{\mathrm dt_n}{t_n}$.
Let $\widetilde{\mathcal K}_{\boldsymbol\gamma}$ be $\mathcal O_{X(\Sigma)}$
endowed with the meromorphic connection with logarithmic pole 
$$\nabla: \mathcal O_{X(\Sigma)}\to \boldsymbol\Omega^1_{X(\Sigma)}(\log D), \quad h\mapsto \mathrm dh+
\Big(\sum_{i=1}^n \gamma_i \frac{\mathrm dt_i}{t_i}\Big)h.$$ Unlike $\mathcal K_{\overline{\boldsymbol\gamma}}$ which 
depends only on $\overline{\boldsymbol\gamma}=\boldsymbol\gamma+\mathbb Z$, the meromorphic connection 
$\widetilde{\mathcal K}_{\boldsymbol\gamma}$ depends on the choice of the representative $\boldsymbol\gamma$ 
of $\overline{\boldsymbol\gamma}$. 

\begin{lemma}\label{prop:logdR} Notation as above. Let $Y$ be a smooth variety, let $\mathcal M$ be 
$\mathcal O_{X(\Sigma)\times Y}$ endowed with an integrable connection, and 
let $$p_1:  X(\Sigma)\times Y\to X(\Sigma)$$ be the 
projection. Suppose none of the residues of the 1-form
$\sum_{i=1}^n \gamma_i \frac{\mathrm dt_i}{t_i}$ along irreducible components 
of $X(\Sigma)-\mathbb G_m^n$ is a strictly positive integer. Then the relative de Rham complex 
of differential forms with logarithmic poles 
$$p_1^* \boldsymbol\Omega^\cdot_{X(\Sigma)}(\log D)\otimes_{\mathcal O_{X(\Sigma)\times Y}} 
p_1^*\widetilde{\mathcal K}_{\boldsymbol\gamma}\otimes_{\mathcal O_{X(\Sigma)\times Y}}\mathcal M$$
for the meromorphic connection 
$p_1^*\widetilde{\mathcal K}_{\boldsymbol\gamma}\otimes_{\mathcal O_{X(\Sigma)\times Y}}\mathcal M$ is quasi-isomorphic 
to the relative de Rham complex 
$$p_1^* \boldsymbol\Omega^\cdot_{X(\Sigma)}\otimes_{\mathcal O_{X(\Sigma)\times Y}} 
p_1^*j_+\mathcal K_{\overline{\boldsymbol\gamma}}\otimes_{\mathcal O_{X(\Sigma)\times Y}}\mathcal M$$
for the $\mathcal D$-module $p_1^*j_+\mathcal K_{\overline{\boldsymbol\gamma}} \otimes_{\mathcal O_{X(\Sigma)\times Y}}\mathcal M$.
\end{lemma}

\begin{proof} We have morphisms 
\begin{eqnarray*}
\boldsymbol\Omega^\cdot_{X(\Sigma)}(\log D)\otimes_{\mathcal O_{X(\Sigma)}} 
\widetilde{\mathcal K}_{\boldsymbol\gamma}&\hookrightarrow &
j_*\boldsymbol\Omega^\cdot_{\mathbb G_m^n}\otimes_{\mathcal O_{X(\Sigma)}} 
\widetilde{\mathcal K}_{\boldsymbol\gamma}
\cong j_*(\boldsymbol\Omega^\cdot_{\mathbb G_m^n}\otimes_{\mathcal O_{\mathbb G_m^n}}
{\mathcal K}_{\overline{\boldsymbol\gamma}})\\
&\cong&  \boldsymbol\Omega^\cdot_{X(\Sigma)}\otimes_{\mathcal O_{X(\Sigma)}} 
j_*\mathcal K_{\overline{\boldsymbol\gamma}}
=\boldsymbol\Omega^\cdot_{X(\Sigma)}\otimes_{\mathcal O_{X(\Sigma)}} 
j_+\mathcal K_{\overline{\boldsymbol\gamma}}\
\end{eqnarray*}
So we have a monomorphism 
$$p_1^* \boldsymbol\Omega^\cdot_{X(\Sigma)}(\log D)\otimes_{\mathcal O_{X(\Sigma)\times Y}} 
p_1^*\widetilde{\mathcal K}_{\boldsymbol\gamma}\otimes_{\mathcal O_{X(\Sigma)\times Y}}\mathcal M
\hookrightarrow p_1^* \boldsymbol\Omega^\cdot_{X(\Sigma)}\otimes_{\mathcal O_{X(\Sigma)\times Y}} 
p_1^*j_+\mathcal K_{\overline{\boldsymbol\gamma}}\otimes_{\mathcal O_{X(\Sigma)\times Y}}\mathcal M.
$$
Let's prove it is a quasi-isomorphism. It suffices to show it is so when restricted to an open subset 
of the form $X_\sigma\cong \mathrm{Spec}\, \mathbb C[\check\sigma\cap \mathbb Z^n]$, where 
$\sigma$ is a cone in $\Sigma$ and $\check\sigma$ is dual cone of $\sigma$. Since $\sigma$ is a regular cone,
after a change of coordinates on $\mathbb G_m^n$, 
the open immersion $j: \mathbb G_m^n\hookrightarrow X_\sigma$ of the open dense torus 
can be identified with 
$$\mathbb G_m^n=\mathbb G_m^r\times\mathbb G_m^{n-r}\hookrightarrow \mathbb A^r\times \mathbb G_m^{n-r}.$$
After the change of coordinates, $\mathcal K_{\overline{\boldsymbol\gamma}}$ becomes 
$\mathcal K_{\overline{\boldsymbol\gamma}'}$ for some $\overline{\boldsymbol\gamma}'\in\mathbb C^n$. 
For convenience, we denote $\boldsymbol\gamma'$ still by $\boldsymbol\gamma$.  

We use Deligne's method in the proof of \cite[Proposition II 3.13]{D}.
To explain the idea of the proof, we first consider the case where $n=1$ and $j: \mathbb G_m^n\hookrightarrow X_\sigma$ 
is the open immersion $j:\mathbb G_m\hookrightarrow \mathbb A^1$. By our assumption, $\mathcal M$ is 
$\mathcal O_{\mathbb A^1\times Y}$ as an $\mathcal O_{\mathbb A^1\times Y}$-module.  
The relative de Rham complex of differential forms with logarithmic poles for the meromorphic connection 
$p_1^*\widetilde{\mathcal K}_{\boldsymbol\gamma}\otimes_{\mathcal O_{X(\Sigma)\times Y}}\mathcal M$ 
is isomorphic to 
\begin{eqnarray}\label{dRlog}
0\to \mathcal O_{\mathbb A^1\times Y} 
\xrightarrow{\mathrm d+(\frac{\gamma}{t} +\phi)\mathrm dt} \mathcal O_{\mathbb A^1\times Y}  \frac{\mathrm dt}{t}\to 0
\end{eqnarray} 
for some $\phi \in\Gamma(\mathbb A^1\times Y, \mathcal O_{\mathbb A^1\times Y})$. 
The de Rham complex for the $\mathcal D$-module 
$p_1^*j_+\mathcal K_{\overline{\boldsymbol\gamma}} \otimes_{\mathcal O_{X(\Sigma)\times Y}}\mathcal M$
is isomorphic to 
\begin{eqnarray}\label{dRmero}
0\to j'_*\mathcal O_{\mathbb G_m\times Y} \xrightarrow{\mathrm d+(\frac{\gamma}{t} +\phi)\mathrm dt} 
j'_* \mathcal O_{\mathbb G_m\times Y} \mathrm dt\to 0,
\end{eqnarray}
where $j': \mathbb G_m\times Y\hookrightarrow \mathbb A^1\times Y$ 
is the base change of $j:\mathbb G_m\hookrightarrow \mathbb A^1$.
Let $D$ be the divisor $\{0\}\times Y$ on $\mathbb A^1\times Y$.
Define an increasing filtration $P$ on the complex (\ref{dRmero}) by 
\begin{eqnarray*}
P_k=\begin{cases}
\Big(0\to \mathcal O_{\mathbb A^1\times Y} ((k+1)D) \xrightarrow{\mathrm d+(\frac{\gamma}{t}+\phi)\mathrm dt} 
\mathcal O_{\mathbb A^1\times Y} ((k+2)D) {\mathrm dt}\to 0\Big)
&\hbox{if }k\geq 0,\\
\Big(0\to 0 \to \mathcal O_{\mathbb A^1\times Y} (D) {\mathrm dt}\to 0\Big)
&\hbox{if }k=-1,\\
0&\hbox{if } k\leq -2.
\end{cases}
\end{eqnarray*}
This filtration induces the following Hodge filtration $F$ on the complex (\ref{dRlog}):
\begin{eqnarray*}
F_k=\begin{cases}
\Big(0\to \mathcal O_{\mathbb A^1\times Y} \xrightarrow{\mathrm d+(\frac{\gamma}{t}+\phi)\mathrm dt} 
\mathcal O_{\mathbb A^1\times Y}\frac{\mathrm dt}{t}\to 0\Big)
&\hbox{if }k\geq 0,\\
\Big(0\to 0 \to  \mathcal O_{\mathbb A^1\times Y}\frac {\mathrm dt}{t}\to 0\Big)
&\hbox{if }k=-1,\\
0&\hbox{if } k\leq -2.
\end{cases}
\end{eqnarray*}
For any $k\geq 1$, we have 
$$P_k/P_{k-1}\cong 
\Big(0\to \frac{1}{t^{k+1}}\mathcal O_{\mathbb A^1\times Y}\Big/\frac{1}{t^{k}}\mathcal O_{\mathbb A^1\times Y} 
\xrightarrow{\frac{\mathrm d}{\mathrm dt}+\frac{\gamma}{t}}\frac{1}{t^{k+2}}
\mathcal O_{\mathbb A^1\times Y}\Big/\frac{1}{t^{k+1}}\mathcal O_{\mathbb A^1\times Y}\to 0\Big).$$
Note that $\phi$ disappears after taking $\mathrm{Gr}^P$. 
Since $\gamma\not=k+1$, the last complex is acyclic. 
We have 
\begin{eqnarray*}
P_0/P_{-1}\cong 
\Big(0\to \frac{1}{t}\mathcal O_{\mathbb A^1\times Y}
\xrightarrow{\frac{\mathrm d}{\mathrm dt}+\frac{\gamma}{t}}\frac{1}{t^2}
\mathcal O_{\mathbb A^1\times Y}\Big/\frac{1}{t}\mathcal O_{\mathbb A^1\times Y}\to 0\Big).
\end{eqnarray*}
Since $\gamma\not=1$, we have $\mathrm{ker}\big(\frac{\mathrm d}{\mathrm dt}+
\frac{\gamma}{t}\big)=\mathcal O_{\mathbb A^1\times Y}$ and 
$\mathrm{coker}\big(\frac{\mathrm d}{\mathrm dt}+\frac{\gamma}{t}\big)=0$.
Finally, we have 
\begin{eqnarray*}
P_{-1}^\cdot/P_{-2}^\cdot\cong 
\Big(0\to 0 \to \mathcal O_{\mathbb A^1\times Y}\frac {\mathrm dt}{t}\to 0\Big),
\end{eqnarray*}
It follows that $\mathrm{Gr}^P$ of the complex (\ref{dRmero}) 
and $\mathrm{Gr}^F$ of the complex (\ref{dRlog}) are quasi-isomorphic. So the complex (\ref{dRmero}) 
and the complex (\ref{dRlog}) are quasi-isomorphic.

Next we treat the higher dimensional case. We have 
$$ p_1^* \boldsymbol\Omega^\cdot_{X_\sigma}\otimes_{\mathcal O_{X_\sigma\times Y}} 
p_1^*j_+\mathcal K_{\overline{\boldsymbol\gamma}}\otimes_{\mathcal O_{X_\sigma\times Y}}\mathcal M
\cong p_1^*\boldsymbol\Omega_{\mathbb A^r\times\mathbb G_m^{n-r}}^\cdot
\otimes_{\mathcal O_{\mathbb A^r\times\mathbb G_m^{n-r}\times Y}} 
j'_*\mathcal O_{\mathbb G_m^r\times\mathbb G_m^{n-r}\times Y},$$
where $j'$ is the open immersion $j': \mathbb G_m^r\times\mathbb G_m^{n-r}\times Y\hookrightarrow
\mathbb A^r\times\mathbb G_m^{n-r}\times Y$. 
Define an increasing filtration $P$ on this complex by 
\begin{align*}
&P_k(p_1^*\boldsymbol\Omega_{\mathbb A^r\times\mathbb G_m^{n-r}}^q\otimes_{\mathcal O_{\mathbb A^r\times\mathbb G_m^{n-r}\times Y}} 
j'_*\mathcal O_{\mathbb G_m^r\times\mathbb G_m^{n-r}\times Y})\\
=& \sum_{\substack{n_1, \ldots, n_r\geq 1\\
n_1+\cdots+n_r\leq k+q+r}}p_1^*\boldsymbol\Omega_{\mathbb A^r\times\mathbb G_m^{n-r}}^q\otimes_{\mathcal O_{\mathbb A^r\times\mathbb G_m^{n-r}\times Y}} 
\mathcal O_{\mathbb A^r\times\mathbb G_m^{n-r}\times Y}(n_1D_1+\cdots+n_rD_r),
\end{align*} where $D_i$ $(i=1, \ldots, r)$ are the divisors on $\mathbb A^r\times\mathbb G_m^{n-r}\times Y$ defined by 
$t_i=0$. 
It induces the Hodge filtration $F$ on the subcomplex 
$p_1^* \boldsymbol\Omega^\cdot_{X_\sigma}(\log D)\otimes_{\mathcal O_{X_\sigma\times Y}} 
p_1^*\widetilde{\mathcal K}_{\boldsymbol\gamma}\otimes_{\mathcal O_{X_\sigma\times Y}}\mathcal M
\cong p_1^*\boldsymbol\Omega_{\mathbb A^r\times\mathbb G_m^{n-r}}^\cdot(\log D)$:
$$
F_k(p_1^*\boldsymbol\Omega_{\mathbb A^r\times\mathbb G_m^{n-r}}^q(\log D))=\begin{cases}
p_1^*\boldsymbol\Omega_{\mathbb A^r\times\mathbb G_m^{n-r}}^q(\log D) &\hbox{if } k\geq -q,\\
0&\hbox{if } k<-q. \end{cases}$$
We have 
\begin{eqnarray*}
\mathrm{Gr}^P(p_1^*\boldsymbol\Omega^\cdot_{\mathbb A^r\times\mathbb G_m^{n-r}}\otimes_
{\mathcal O_{\mathbb A^r\times\mathbb G_m^{n-r}\times Y}}j'_*\mathcal O_{\mathbb G_m^r\times\mathbb G_m^{n-r}\times Y})
&\cong&\mathrm{Gr}^P (\boldsymbol\Omega^\cdot_{\mathbb A^1}\otimes_{\mathcal O_{\mathbb A^1}}j_*\mathcal O_{\mathbb G_m})
^{\boxtimes r} \boxtimes
\boldsymbol\Omega^\cdot_{\mathbb G_m^{n-r}}\boxtimes\mathcal O_Y,\\
\mathrm{Gr}^F(p_1^*\boldsymbol\Omega_{\mathbb A^r\times\mathbb G_m^{n-r}}^\cdot(\log D))
&\cong&\mathrm{Gr}^F (\boldsymbol\Omega^\cdot_{\mathbb A^1}(\log D))^{\boxtimes r} \boxtimes
\boldsymbol\Omega^\cdot_{\mathbb G_m^{n-r}}\boxtimes \mathcal O_Y.
\end{eqnarray*}
Combined with the dimension 1 case treated above, we conclude that the monomorphism 
$$p_1^* \boldsymbol\Omega^\cdot_{X_\sigma}(\log D)\otimes_{\mathcal O_{X_\sigma\times Y}} 
p_1^*\widetilde{\mathcal K}_{\boldsymbol\gamma}\otimes_{\mathcal O_{X_\sigma\times Y}}\mathcal M
\hookrightarrow p_1^* \boldsymbol\Omega^\cdot_{X_\sigma}\otimes_{\mathcal O_{X_\sigma\times Y}} 
p_1^*j_+\mathcal K_{\overline{\boldsymbol\gamma}}\otimes_{\mathcal O_{X_\sigma\times Y}}\mathcal M.
$$
induces a quasi-isomorphism
$$\mathrm{Gr}^F(p_1^* \boldsymbol\Omega^\cdot_{X_\sigma}(\log D)\otimes_{\mathcal O_{X_\sigma\times Y}} 
p_1^*\widetilde{\mathcal K}_{\boldsymbol\gamma}\otimes_{\mathcal O_{X_\sigma\times Y}}\mathcal M)\to 
\mathrm{Gr}^P(p_1^* \boldsymbol\Omega^\cdot_{X_\sigma}\otimes_{\mathcal O_{X_\sigma\times Y}} 
p_1^*j_+\mathcal K_{\overline{\boldsymbol\gamma}}\otimes_{\mathcal O_{X_\sigma\times Y}}\mathcal M).$$ Hence the 
above monomorphism is a quasi-isomorphism.
\end{proof}

\subsection{Proof of Proposition \ref{prop:smalldR}} Let $\Sigma_0$ be the fan of faces of the dual cone $\check\delta$ of
$\delta$. The toric variety associated to $\Sigma_0$ is 
$X(\Sigma_0)\cong\mathrm{Spec}\,\mathbb C[\delta\cap \mathbb Z^n]$. 
Let $\Sigma$ be a regular refinement of $\Sigma_0$, let $X(\Sigma)$ the toric variety associated to 
$\Sigma$, and let $$\mathbb R_{\geq 0}l_1, \ldots, 
\mathbb R_{\geq 0}l_m$$ be all the one dimensional cones in 
$\Sigma$, where $$l_i(\mathbf w)=a_{i1}w_1+\cdots+a_{in}w_n\quad (i=1,\ldots, m)$$ are linear forms on $\mathbb Z^n$
with relatively prime integer coefficients $a_{i1},\ldots, a_{in}$. Choose a representative $\boldsymbol\gamma$
of $\overline{\boldsymbol\gamma}$ so that the following condition holds:
\begin{eqnarray}\label{condition}
l_1(\boldsymbol\gamma), \ldots, l_m(\boldsymbol\gamma)\hbox{ are not strictly positive integers.}
\end{eqnarray}
Let us prove Proposition \ref{prop:smalldR} for such a choice of $\boldsymbol\gamma$.
Let $$F_0: X(\Sigma_0)\times\mathbb A^N\to \mathbb A^1$$ be the morphism defined by the Laurent polynomial 
$F(\mathbf t, \mathbf x)=\sum_{j=1}^N x_j \mathbf t^{\mathbf w_j}\in \mathbb C[\mathbf x][\delta\cap \mathbb Z^n]$,
let $$p: X(\Sigma)\to X(\Sigma_0)$$ be the canonical morphism which is a resolution of singularities, and let 
$\overline F=F_0 (p\times\mathrm{id}_{\mathbb A^N})$. Fix notation by the following commutative diagram
$$\begin{tikzcd}
\mathbb G^n_m\arrow[d,hook, "j",swap]&\mathbb G_m^n\times \mathbb A^N\arrow[l, "\pi_1",swap]
\arrow[d,hook, "j'",swap]\arrow[dr, "\pi_2"]&\\
X(\Sigma)&X(\Sigma)\times \mathbb A^N\arrow[l,"p_1",swap]\arrow[r,"p_2"]\arrow[d, "p\times\mathrm{id}_{\mathbb A^N}"]
\arrow[dd,bend right=70,"\overline F", swap]&\mathbb A^N\\
&X(\Sigma_0)\times\mathbb A^N\arrow[ur, "q_2",swap]\arrow[d,"F_0"]&\\
&\mathbb A^1.&
\end{tikzcd}$$
Since $\pi_1$ and $p_1$ are smooth, we have
\begin{align*}
\mathrm{Hyp}_{\overline{\boldsymbol\gamma}}&\cong 
R\pi_{2+}(\pi_1^*\mathcal K_{\overline{\boldsymbol\gamma}}\otimes_{\mathcal O_{\mathbb G_m^n\times \mathbb A^N}} F^*\mathcal L)
\cong Rp_{2+}Rj'_+(\pi_1^*\mathcal K_{\overline{\boldsymbol\gamma}}
\otimes_{\mathcal O_{\mathbb G_m^n\times \mathbb A^N}} j'^*\overline F^*\mathcal L)\\
&\cong Rp_{2+}(Rj'_+\pi_1^*\mathcal K_{\overline{\boldsymbol\gamma}}\otimes_{\mathcal O_{X(\Sigma)\times \mathbb A^N}} 
\overline F^*\mathcal L) \qquad\hbox{(projection formula)}\\
&\cong Rp_{2+}(p_1^* Rj_+\mathcal K_{\overline{\boldsymbol\gamma}}\otimes_{\mathcal O_{X(\Sigma)\times \mathbb A^N}} 
\overline F^*\mathcal L)\qquad(\hbox{base change theorem, } \pi_1^*=R\pi_1^![-N]\,\; p_1^*=Rp_1^![-N])\\
&\cong Rp_{2+}(p_1^* j_+\mathcal K_{\overline{\boldsymbol\gamma}}\otimes_{\mathcal O_{X(\Sigma)\times \mathbb A^N}} 
\overline F^*\mathcal L)\\ 
&\cong  Rp_{2*}\big(\mathcal D_{\mathbb A^N\leftarrow X(\Sigma)\times\mathbb A^N}\otimes^L_{\mathcal D_{X(\Sigma)\times
\mathbb A^N}}
(p_1^* j_+\mathcal K_{\overline{\boldsymbol\gamma}}\otimes_{\mathcal O_{X(\Sigma)\times \mathbb A^N}} 
\overline F^*\mathcal L)\big).
\end{align*}
By \cite[Lemma 1.5.27]{Hotta}, in the derived category of 
$(p_2^{-1}\mathcal D_{\mathbb A^N}, \mathcal D_{X(\Sigma)\times\mathbb A^N})$-modules, we have 
$$\mathcal D_{\mathbb A^N\leftarrow X(\Sigma)\times\mathbb A^N}
\cong \omega_{X(\Sigma)}\boxtimes \mathcal D_{\mathbb A^N}
\cong(\boldsymbol\Omega^\cdot_{X(\Sigma)}[n]\otimes_{\mathcal O_{X(\Sigma)}}\mathcal D_{X(\Sigma)})\boxtimes 
\mathcal D_{\mathbb A^N}.$$
So 
$\mathrm{Hyp}_{\overline{\boldsymbol\gamma}}$ is isomorphic to the relative de Rham complex 
$$Rp_{2*}(p_1^* \boldsymbol\Omega^\cdot_{X(\Sigma)}\otimes_{\mathcal O_{X(\Sigma)\times \mathbb A^N}} 
p_1^*j_+\mathcal K_{\overline{\boldsymbol\gamma}}\otimes_{\mathcal O_{X(\Sigma)\times \mathbb A^N}} \overline F^*\mathcal L)[n].$$
The residues of $\sum_{i=1}^n \gamma_i \frac{\mathrm dt_i}{t_i}$ along irreducible components
of $X(\Sigma)-\mathbb G_m^n$ are exactly $l_1(\boldsymbol\gamma), \ldots, l_m(\boldsymbol\gamma)$. By the condition 
(\ref{condition}), Lemma \ref{prop:logdR} is applicable, and we get 
\begin{align}\label{aln:HypdR}
\mathrm{Hyp}_{\overline{\boldsymbol\gamma}}
\cong Rp_{2*}(p_1^* \boldsymbol\Omega^\cdot_{X(\Sigma)}(\log D)\otimes_{\mathcal O_{X(\Sigma)\times \mathbb A^N}} 
p_1^*\widetilde{\mathcal K}_{\boldsymbol\gamma}\otimes_{\mathcal O_{X(\Sigma)\times \mathbb A^N}}\overline F^*\mathcal L)[n].
\end{align}
As a normal toric variety, $X(\Sigma_0)$ has rational singularity (\cite[Theorems 11.4.2, 9.2.5]{CLS}) and we have 
$$Rp_*\mathcal O_{X(\Sigma)}\cong \mathcal O_{X(\Sigma_0)}.$$
So we have 
$$R (p\times\mathrm{id}_{\mathbb A^N})_*\mathcal O_{X(\Sigma)\times\mathbb A^N}
\cong \mathcal O_{X(\Sigma_0)\times\mathbb A^N}.$$
Let $$q_2: X(\Sigma_0)\times \mathbb A^N\to \mathbb A^N$$ be the projection. Since  $X(\Sigma_0)$
is affine, $q_2$ is an affine morphism. So we have 
\begin{align*}
Rp_{2 *}\mathcal O_{X(\Sigma)\times\mathbb A^N}\cong Rq_{2*}
R(p\times\mathrm{id}_{\mathbb A^N})_*\mathcal O_{X(\Sigma)\times\mathbb A^N}
\cong q_{2*}\mathcal O_{X(\Sigma_0)\times\mathbb A^N}.
\end{align*}
Since $\Omega^1_{X(\Sigma)}(\log D)$ is a free $\mathcal O_{X(\Sigma)}$-module 
with basis $\frac{\mathrm dt_1}{t_1}, \ldots,\frac{\mathrm dt_n}{t_n}$, and 
$p_1^*\widetilde{\mathcal K}_{\boldsymbol\gamma}\otimes_{\mathcal O_{X(\Sigma)\times \mathbb A^N}}\overline F^*\mathcal L$
is isomorphic to $\mathcal O_{X(\Sigma)\times \mathbb A^N}$ as an $\mathcal O_{X(\Sigma)\times \mathbb A^N}$-module, we have 
$$Rp_{2*}(p_1^* \boldsymbol\Omega^\cdot_{X(\Sigma)}(\log D)\otimes_{\mathcal O_{X(\Sigma)\times \mathbb A^N}} 
p_1^*\widetilde{\mathcal K}_{\boldsymbol\gamma}\otimes_{\mathcal O_{X(\Sigma)\times \mathbb A^N}}\overline F^*\mathcal L)
\cong \mathcal C_{\boldsymbol\gamma}^\cdot.$$
Combined with (\ref{aln:HypdR}), we get 
$\mathrm{Hyp}_{\overline{\boldsymbol\gamma}}\cong \mathcal C^\cdot_{\boldsymbol\gamma}[n]$. \qed

\subsection{}
Let 
$$C(A)=\{c_1 \mathbf w_1+\cdots+c_N\mathbf w_N: c_i\in \mathbb Z_{\geq 0}\}$$
be the semigroup generated by $\mathbf w_1,\ldots,\mathbf w_N$, let
$$\mathbb C[C(A)]:= \Big\{\sum_{\mathbf w\in C(A)} a_{\mathbf w} \mathbf t^{\mathbf w}: 
a_{\mathbf w}\in \mathbb C\Big\}$$ be the semigroup ring, let $q'_2: \mathrm{Spec}\,\mathbb C[C(A)]\times\mathbb A^N\to\mathbb A^N$
be the projection, and let $(\mathcal E^\cdot_{\boldsymbol\gamma},d)$ be the complex of de Rham type so that
$$\mathcal E_{\boldsymbol\gamma}^q = \bigoplus_{1\leq i_1<\cdots<i_q\leq n} q'^*_2 \mathcal O_{\mathrm{Spec}\, 
\mathbb C[C(A)]\times\mathbb A^N}
\frac{\mathrm dt_{i_1}}{t_{i_1}}\wedge\cdots\wedge \frac{\mathrm dt_{i_q}}{t_{i_q}}.$$
and the operator $d^q: \mathcal E_{\boldsymbol\gamma}^q\to \mathcal E_{\boldsymbol\gamma}^{q+1}$ is given by 
(\ref{aln:d}). This is a complex of $\mathcal D_{\mathbb A^N}$-modules so that the action of 
$\partial_{x_j}$ is given by the formula (\ref{aln:nabla}). We have the following result of Adolphson \cite[Theorem 4.4]{A}. 

\begin{lemma}[Adolphson] \label{lm: AGKZdeRham} The map
\begin{eqnarray*}
\mathcal D_{\mathbb A^N}&\to&  q'^*_2 \mathcal O_{\mathrm{Spec}\, 
\mathbb C[C(A)]\times\mathbb A^N}
\frac{\mathrm dt_1}{t_1}\wedge\cdots\wedge \frac{\mathrm dt_n}{t_n},\\
\sum_{i_1\ldots i_N} a_{i_1\ldots i_N}(\mathbf x)\partial_{x_1}^{i_1}\cdots\partial_{x_N}^{i_N}&\mapsto& 
\sum_{i_1\ldots i_N} a_{i_1\ldots i_N}(\mathbf x) \mathbf t^{i_1\mathbf w_1+\cdots+i_N\mathbf w_N}\frac{\mathrm dt_1}{t_1}\wedge\cdots\wedge \frac{\mathrm dt_n}{t_n}
\end{eqnarray*}
induces an isomorphism of $\mathcal D_{\mathbb A^N}$-module
$$\mathcal H_{\boldsymbol\gamma}= \mathcal D_{\mathbb A^N}\Big/\mathcal D_{\mathbb A^N}
\Big(\sum_{j=1}^N w_{ij}x_j\partial_{x_j}+\gamma_i,\,\square_{\boldsymbol\lambda}\Big)\cong 
\mathcal H^n(\mathcal E^\cdot_{\boldsymbol\gamma}).$$ 
\end{lemma}

\subsection{Proof of Proposition \ref{prop:betterGKZ}} 
Via the isomorphism in Lemma \ref{lm: AGKZdeRham},
the transition morphism
$$\cdot\partial_{x_k}: \mathcal H_{\boldsymbol\gamma+\mathbf v+\mathbf w_k}\to \mathcal H_{\boldsymbol\gamma+\mathbf v}$$
can be lifted to the chain map
\begin{eqnarray*}
\mathbf t^{\mathbf w_k}: \mathcal E^\cdot_{\boldsymbol\gamma+\mathbf v+\mathbf w_k}\to 
\mathcal E^\cdot_{\boldsymbol\gamma+\mathbf v}.
\end{eqnarray*} 
We thus get a direct system of complexes
$$(\mathcal E^\cdot_{\boldsymbol\gamma+\mathbf v}, \mathbf t^{\mathbf w_1}, 
\ldots, \mathbf t^{\mathbf w_N})_{\mathbf v\in \delta\cap \mathbb Z^n}.$$
We have 
\begin{eqnarray}\label{eqn:HE}
\varinjlim_{\mathbf v\in \delta\cap\mathbb Z^n}\mathcal H_{\boldsymbol\gamma+\mathbf v}
&\cong&\varinjlim_{\mathbf v\in \delta\cap\mathbb Z^n} \mathrm{coker}(d^{n-1}: \mathcal E^{n-1}_{\boldsymbol\gamma+\mathbf v}\to 
\mathcal E^n_{\boldsymbol\gamma+\mathbf v})\nonumber\\
&\cong& \mathrm{coker}(d^{n-1}: \varinjlim_{\mathbf v\in \delta\cap\mathbb Z^n} \mathcal E^{n-1}_{\boldsymbol\gamma+\mathbf v}\to 
\varinjlim_{\mathbf v\in \delta\cap\mathbb Z^n} \mathcal E^n_{\boldsymbol\gamma+\mathbf v})\nonumber\\
&\cong& \mathcal H^n( \varinjlim_{\mathbf v\in \delta\cap\mathbb Z^n} \mathcal E^\cdot_{\boldsymbol\gamma+\mathbf v}).
\end{eqnarray}
We need to be careful with the direct limit since $(\delta\cap\mathbb Z^n,\leq)$ is not a direct set. We use the fact 
that taking direct limit commutes with taking cokernel, which follows from the universal properties of the 
direct limit and the cokernel.
For each $\mathbf v\in \delta\cap\mathbb Z^n$, we have a chain map 
$$\mathbf t^{\mathbf v}: \mathcal E^\cdot_{\boldsymbol\gamma+\mathbf v}\to \mathcal C^\cdot_{\boldsymbol\gamma}.$$
It is a monomorphism, and its image consists of differential forms 
(with basis $\{\frac{\mathrm dt_{i_1}}{t_{i_1}}\wedge\cdots\wedge 
\frac{\mathrm dt_{i_q}}{t_{i_q}}\}$) 
whose coefficients are linear combinations
over $\mathcal O_{\mathbb A^N}$ of monomials of the form $\mathbf t^{\mathbf w}$ so that $\mathbf w\in \mathbf v+ C(A)$. 
It is compatible with the transition morphisms and hence we have a chain map
$$\varinjlim_{\mathbf v\in \delta\cap\mathbb Z^n} \mathcal E^\cdot_{\boldsymbol\gamma+\mathbf v}\to 
\mathcal C^\cdot_{\boldsymbol\gamma}.$$
This chain map is a monomorphism. It is surjective since
$$\delta\cap\mathbb Z^n=\bigcup_{\mathbf v\in \delta\cap\mathbb Z^n}(\mathbf v+ C(A)).$$ So we have
$$\varinjlim_{\mathbf w\in \delta\cap\mathbb Z^n} \mathcal E^\cdot_{\boldsymbol\gamma+\mathbf v}\cong 
\mathcal C^\cdot_{\boldsymbol\gamma}.$$
Combined with (\ref{eqn:HE}), we get $\varinjlim_{\mathbf w\in \delta\cap\mathbb Z^n} \mathcal H_{\boldsymbol\gamma+\mathbf v}\cong
\mathcal H^n(\mathcal C^\cdot_{\boldsymbol\gamma}).$ \qed

\section{Proof of Proposition \ref{prop:goodR}}

Let $U$ be an open subset of $\mathbb A^N$ such that for any rational point 
$\mathbf a$ in $\mathbb A^N$, the Laurent polynomial $F(\mathbf t, \mathbf a)$ is nondegenerate with 
respect to $\Delta_\infty$. In the following, $\mathcal K$ 
is either the coordinate ring $\mathcal O_V(V)$
of an affine open subset $V$ of $U$, or the residue field $k(\mathbf a)\cong\mathbb C$
of $\mathcal O_U$ at a rational point $\mathbf a$. Recall that $\delta$ is the cone generated by $\mathbf w_1,\ldots, \mathbf w_N$.
Let 
$$\mathcal R=\mathcal K[\delta\cap\mathbb Z^n].$$
For any $\mathbf w\in \delta$, define 
$$\rho(\mathbf w)=\inf\{r\in \mathbb R_{\geq 0}: \mathbf w\in r\Delta_\infty\}.$$
We have 
$$\rho(r\mathbf w)=r\rho(\mathbf w),\quad \rho(\mathbf w_1+\mathbf w_2)\leq \rho(\mathbf w_1)+\rho(\mathbf w_2).$$
Let $M$ be the smallest positive integer such that $$\rho(\delta\cap\mathbb Z^n)\subset \frac{1}{M}\mathbb Z,$$
and let $\mathcal R_p$ be the $\mathcal K$-submodule of $\mathcal R$ generated by
monomials $\mathbf t^{\mathbf w}$ with $\rho (\mathbf w)\leq \frac{p}{M}$. We have 
$$\mathcal R_p \mathcal R_q\subset \mathcal R_{p+q}.$$ We thus get a filtered ring $\mathcal R$. 
Let $R=\mathrm{Gr}(\mathcal R)$ be the corresponding graded ring. 
As an abelian group, $R$ and $\mathcal R$ are isomorphic. But in $R$, we have 
$\mathbf t^{\mathbf w_1}\mathbf t^{\mathbf w_2}=
\mathbf t^{\mathbf w_1+\mathbf w_2}$ if there exists a face $\Gamma$ of $\Delta_\infty$ not containing 
the origin such that both $\mathbf w_1$ and $\mathbf w_2$ lie in the cone $\mathbb R_{\geq 0}\Gamma$ 
generated by $\Gamma$, and $\mathbf t^{\mathbf w_1}\mathbf t^{\mathbf w_2}=0$  otherwise. 

Let $(C^\cdot,d)$ the complex so that  
$$C^q=
\bigoplus_{1\leq i_1<\cdots< i_q\leq n} \mathcal K[\delta\cap \mathbb Z^n] \frac{\mathrm dt_{i_1}}{t_{i_1}}\wedge\cdots\wedge 
\frac{\mathrm dt_{i_q}}{t_{i_q}},$$ and $d:C^q\to C^{q+1}$ is given by the formula (\ref{aln:d}).
In the case where $\mathcal K=\mathcal O_V(V)$, we have 
$$C\cong\Gamma(V, \mathcal C_{\boldsymbol\gamma}^\cdot).$$
In the case where $\mathcal K=k(\mathbf a)$, we have 
\begin{align*}
&\;d\Big(h\frac{\mathrm dt_{i_1}}{t_{i_1}}\wedge\cdots\wedge 
\frac{\mathrm dt_{i_q}}{t_{i_q}}\Big)= \sum_{i=1}^n
\Big(t_i\frac{\partial h}{\partial t_i}+\gamma_i h+ \sum_j a_jw_{ij} \mathbf t^{\mathbf w_j} h\Big)\frac{\mathrm dt_i}{t_i}\wedge
\frac{\mathrm dt_{i_1}}{t_{i_1}}\wedge\cdots\wedge 
\frac{\mathrm dt_{i_q}}{t_{i_q}}
\\
& \qquad\qquad=\Big(\mathbf t^{-\boldsymbol \gamma}e^{-F(\mathbf t, \mathbf a)}\circ 
\mathrm d_{\mathbf t} \circ \mathbf t^{\boldsymbol \gamma}e^{F(\mathbf t, \mathbf a)}\Big)
\Big(h\frac{\mathrm dt_{i_1}}{t_{i_1}}\wedge\cdots\wedge 
\frac{\mathrm dt_{i_q}}{t_{i_q}}\Big).
\end{align*}
The second case is just the base change of the first case with respect to 
$$\Gamma(V,\mathcal O_V)\to \mathbb C,\quad g\mapsto g(\mathbf a).$$
Equip $C^\cdot$ with the increasing filtration $F_pC^\cdot$ so that 
$$F_p C^q=\bigoplus_{1\leq i_1<\cdots< i_q\leq n} \mathcal R_{p+Mq} \frac{\mathrm dt_{i_1}}{t_{i_1}}\wedge\cdots\wedge 
\frac{\mathrm dt_{i_q}}{t_{i_q}}.$$
The operator $d: C^q\to C^{q+1}$ preserves this filtration. We have 
$$\mathrm{Gr}(C^q)=
\bigoplus_{1\leq i_1<\cdots< i_q\leq n} R \frac{\mathrm dt_{i_1}}{t_{i_1}}\wedge\cdots\wedge 
\frac{\mathrm dt_{i_q}}{t_{i_q}}.$$ 
In the case where $\mathcal K=\mathcal O_V(V)$ (resp. $\mathcal K=k(\mathbf a)$), the operator $d$ for the complex 
$\mathrm{Gr}(C^\cdot)$ is given by 
\begin{align*}
\;d\Big(h\frac{\mathrm dt_{i_1}}{t_{i_1}}\wedge\cdots\wedge 
\frac{\mathrm dt_{i_q}}{t_{i_q}}\Big)=& \sum_{i=1}^n \sum_{\mathbf w_j\in\partial \Delta_\infty} x_jw_{ij} \mathbf t^{\mathbf w_j} h
\frac{\mathrm dt_i}{t_i}\wedge\frac{\mathrm dt_{i_1}}{t_{i_1}}\wedge\cdots\wedge 
\frac{\mathrm dt_{i_q}}{t_{i_q}}\\
\Big(\hbox{resp. }&\sum_{i=1}^n \sum_{\mathbf w_j\in\partial \Delta_\infty} a_jw_{ij} \mathbf t^{\mathbf w_j} h
\frac{\mathrm dt_i}{t_i}\wedge\frac{\mathrm dt_{i_1}}{t_{i_1}}\wedge\cdots\wedge 
\frac{\mathrm dt_{i_q}}{t_{i_q}}\Big).
\end{align*}
Thus $\mathrm{Gr}(C^\cdot)$ is the Koszul complex $K^\cdot(R)$ for the graded ring $R$ with respect the image of the
sequence $t_1\frac{\partial f}{\partial t_1},\ldots,  t_n\frac{\partial f}{\partial t_n}$
in $\mathcal R_M/\mathcal R_{M-1}$, where 
$$f(\mathbf t)=\sum_{\mathbf w_j\in \partial \Delta_\infty} x_j \mathbf t^{\mathbf w_j}\quad 
(\hbox{resp. }\sum_{\mathbf w_j\in \partial \Delta_\infty} a_j \mathbf t^{\mathbf w_j}).$$
Recall that for any $R$-module $N$, the Koszul complex $K^\cdot(N)$ 
with respect to the sequence $t_1 \frac{\partial f}{\partial t_1},\ldots, t_n \frac{\partial f}{\partial t_n}$ can be defined by 
\begin{align*}
&K^q(N)=
\bigoplus_{1\leq i_1<\cdots< i_q\leq n} N \frac{\mathrm dt_{i_1}}{t_{i_1}}\wedge\cdots\wedge 
\frac{\mathrm dt_{i_q}}{t_{i_q}},\\
&\;d\Big( s\frac{\mathrm dt_{i_1}}{t_{i_1}}\wedge\cdots\wedge 
\frac{\mathrm dt_{i_q}}{t_{i_q}} \Big)=\sum_{i=1}^n \Big(t_i \frac{\partial f}{\partial t_i}\Big) s
\frac{\mathrm dt_i}{t_i}\wedge \frac{\mathrm dt_{i_1}}{t_{i_1}}\wedge\cdots\wedge 
\frac{\mathrm dt_{i_q}}{t_{i_q}}.
\end{align*}

\begin{lemma}\label{lm:finiteD} We have $H^i(C^\cdot)=0$ for $i\not =n$ and 
$H^n(C^\cdot)$ is a finitely generated $\mathcal K$-module. In the case where $\mathcal K=k(\mathbf a)$, 
we have $\mathrm{dim}\, H^n(C^\cdot)=n!\mathrm{vol}(\Delta_\infty)$
\end{lemma} 

\begin{proof} We have a regular spectral sequence 
$$E_1^{pq}=H^{p+q}(F_{-p}C^\cdot/F_{-p-1}C^\cdot)
\Rightarrow H^{p+q}(C^\cdot).$$
We have 
$\bigoplus_p E_1^{p, i-p}\cong H^i(K^\cdot(R))$. 
Our assertion follows from Lemmas \ref{prop:Kouchnirenko} and \ref{lm:rank} below.
\end{proof}

\begin{lemma} \label{prop:Kouchnirenko} We have $H^i(K^\cdot (R))=0$ for $i\not =n$, and $H^n(K^\cdot (R))$ is
a finitely generated $\mathcal K$-module. 
\end{lemma} 

\begin{proof} We use Kouchnirenko's method in \cite[\S2]{K}. 
For each face $\Gamma$ of $\Delta_\infty$ not containing the origin, let 
$$R_\Gamma=\mathcal K[\mathbb R_{\geq 0}\Gamma \cap \mathbb Z^n].$$ 
For each face of $\Gamma'$ of $\Gamma$ of codimension 1, we have a $\mathcal K$-algebra homomorphism 
$$p_{\Gamma\Gamma'}: R_\Gamma\to R_{\Gamma'}, \quad\mathbf t^{\mathbf w}\mapsto 
\begin{cases}
\mathbf t^{\mathbf w}&\hbox{if }\mathbf w\in \mathbb R_{\geq 0}\Gamma'\cap\mathbb Z^n,\\
0&\hbox{otherwise}.
\end{cases}$$
Define a complex $A^\cdot$ as follows. 
For each $0\leq p\leq n-1$, let $I_p$ be the set of those $p$ dimensional faces $\Gamma$ of $\Delta_\infty$ such that
$\Gamma$ does not lie on any $n-1$ dimensional face of $\Delta_\infty$ containing the origin. For any $0\leq q\leq n-1$, set 
$$A^q=\bigoplus_{\Gamma\in I_{n-1-q}} R_\Gamma.$$
Fix an orientation for each face of $\Delta_\infty$ not containing the origin. Define $d^q: A^q\to A^{q+1}$ so that for any 
$\Gamma\in I_{n-1-q}$ and $\Gamma'\in I_{n-2-q}$, 
the composite $$R_\Gamma\hookrightarrow A^q\stackrel{d^q}\to A^{q+1}\twoheadrightarrow R_{\Gamma'}$$
is $0$ if $\Gamma'$ is not a face of $\Gamma$, 
and is $p_{\Gamma\Gamma'}$ (resp. $-p_{\Gamma\Gamma'}$) if 
$\Gamma'$ is a face of $\Gamma$ and the orientation of $\Gamma'$ is the same as (resp. opposite to) the one induced from 
the orientation of $\Gamma$, where $R_\Gamma\hookrightarrow A^q$ 
(resp. $A^{q+1}\twoheadrightarrow R_{\Gamma'}$) is the canonical inclusion 
(resp. projection). We thus get a complex 
$$0\to A^0\stackrel {d^0}\to A^1\stackrel {d^1}\to\cdots \stackrel {d^{n-2}}\to A^{n-1}\to 0.$$
We take $A^\cdot$ to be this complex if $0$ lies on the boundary of $\Delta_\infty$. 
If $0$ lies in the interior of $\Delta_\infty$, we take $A^\cdot$ to be the augmented complex 
$$0\to A^0\stackrel {d^0}\to A^1\stackrel {d^1}\to\cdots \stackrel {d^{n-2}}\to A^{n-1}\stackrel{\epsilon}\to\mathcal K\to 0,$$
where for any $\Gamma\in I_0$, the restriction of $\epsilon$ to $R_\Gamma$ is the $\mathcal K$-algebra homomorphism
$$R_\Gamma\to\mathcal K, \quad \mathbf t^{\mathbf w}\mapsto \begin{cases}
1&\hbox{if }\mathbf w=0, \\
0&\hbox{otherwise}.
\end{cases}$$

\medskip
\emph{Claim}: We have $H^q(A^\cdot)=0$ for all $q\not=0$. 
\medskip

For any $\mathbf w\in \delta\cap \mathbb Z^n$ and any face $\Gamma$ of $\Delta_\infty$ not containing the 
orgin, let 
$$R_{\Gamma}(\mathbf w)=\begin{cases}
\mathcal K\mathcal \mathbf t^{\mathbf w}&\hbox{if }\mathbf w\in \mathbb R_{\geq 0}\Gamma\cap\mathbb Z^n,\\
0&\hbox{otherwise}.
\end{cases}$$
For any $0\leq q\leq n-1$, let $$A^q(\mathbf w)=\bigoplus_{\Gamma\in I_{n-1-q}} R_\Gamma(\mathbf w).$$
If $0$ lies in the interior of $\Delta_\infty$, let 
$$A^n(\mathbf w)=\begin{cases}
\mathcal K&\hbox{if } \mathbf w=0, \\
0&\hbox{otherwise}.
\end{cases}$$
Then $A^\cdot(\mathbf w)$ is a subcomplex of $A^\cdot$ and 
$$A^\cdot=\bigoplus_{\mathbf w\in\delta\cap \mathbb Z^n}A^\cdot(\mathbf w).$$
Let $X(\mathbf w)$ be the union of those faces $\Gamma$ in $\bigcup_p I_p$ such that 
$\mathbf w\in \mathbb R_{\geq 0}\Gamma $. If $0$ lies in the interior of $\Delta_\infty$ and $\mathbf w=0$, 
then $A^\cdot(\mathbf w)$ is isomorphic to the augmented 
cellular chain complex (with coefficient $\mathcal K$) of $X(\mathbf w)$ after a degree shifting, and 
$X(\mathbf w)$ is homeomorphic to an $n-1$ dimensional sphere. So for any $q\not=0$, we have 
$$H^q(A^\cdot(\mathbf w))\cong 
\tilde H_{n-1-q}(X(\mathbf w),\mathcal K)=0.$$
If $0$ lies in the interior but $\mathbf w\not=0$, or if $0$ lies on the boundary of $\Delta_\infty$ 
and $\mathbf w\in \delta\cap \mathbb Z^n$ is arbitrary, then 
$A^\cdot(\mathbf w)$ is isomorphic to the relative cellular chain complex
of the pair $(X(\mathbf w), \partial X(\mathbf w))$ 
after a degree shifting, and the pair $(X(\mathbf w), \partial X(\mathbf w))$ is homeomorphic to a pair of 
an $n-1$ dimensional ball and its boundary sphere. For any $q\not=0$, we have 
$$H^q(A^\cdot(\mathbf w))\cong 
H_{n-1-q}\Big((X(\mathbf w), \partial X(\mathbf w)),\mathcal K\Big)=0.$$
Our claim follows.  

\medskip
For each $\Gamma\in I_{n-1}$, we have a ring 
homomorphism 
$$p_\Gamma: R\to R_{\Gamma}, \quad\mathbf t^{\mathbf w}\mapsto 
\begin{cases}
\mathbf t^{\mathbf w}&\hbox{if }\mathbf w\in \mathbb R_{\geq 0}\Gamma\cap\mathbb Z^n,\\
0&\hbox{otherwise}.
\end{cases}$$ Fix an orientation for $\Delta_\infty$.
We have a monomorphism
$$R\to \bigoplus_{\Gamma\in I_{n-1}} R_\Gamma$$
so that its composite with the projection to $R_\Gamma$  
is $p_\Gamma$ (resp. $-p_\Gamma$) if the orientation of $\Gamma$ is the same as (resp. opposite to) 
the one induced from that of $\Delta_\infty$. 
This monomorphism identifies $R$ with the kernel of $A^0\to A^1$. So 
$A^\cdot$ is a resolution of $R$. 

Consider the bicomplex $K^{\cdot\cdot}$ defined by $K^{pq}= K^q(A^p)$, where $K^\cdot(A^p)$ is the Koszul complex
for the $R$-module $A^p$. Note that $K^q(A^\cdot)$ is a resolution of $K^q(R)$ for each $q$. 
So we have a biregular spectral sequence 
$$E_1^{pq}=H^q(K^\cdot(A^p)) \Rightarrow H^{p+q}(K^\cdot(R)).$$ 
By the definition of the Koszul complex, we have $H^i(K^\cdot (R))=0$ for $i>n$. 
To prove the lemma, it suffices to show $H^q(K^\cdot(A^p))=0$ for $p+q< n$ and 
$H^q(K^\cdot(A^p))$ is a finitely generated $\mathcal K$-module for $p+q=n$. This follows from Lemma \ref{lm:facial} below.
\end{proof}

\begin{lemma}\label{lm:facial} For any $n-1-p$ dimensional face $\Gamma$ of $\Delta_\infty$ not containing the origin, 
we have $H^q(K^\cdot(R_\Gamma))=0$ for $q< n-p$, and $H^{n-p}(K^\cdot(R_\Gamma))$ is a finitely generated $\mathcal K$-module. 
\end{lemma}

\begin{proof}
In the case where $\mathcal K=\mathcal O_V(V)$ (resp. $\mathcal K=k(\mathbf a)$), let
$$f_\Gamma=\sum_{\mathbf w_j\in \Gamma}x_j\mathbf t^{\mathbf w_j}\quad 
(\hbox{resp. } \sum_{\mathbf w_j\in \Gamma}a_j\mathbf t^{\mathbf w_j}).$$
Then $K^\cdot(R_\Gamma)$ is the Koszul complex 
for the ring $R_\Gamma$ with respect to 
the sequence $t_1\frac{\partial f_\Gamma}{\partial t_1},\ldots, t_n\frac{\partial f_\Gamma}{\partial t_n}$. 
Consider the toric $\mathcal K$-scheme $\mathrm{Spec}\, R_\Gamma=
\mathrm{Spec}\, \mathcal K[\mathbb R_{\geq 0}\Gamma \cap \mathbb Z^n]$, on which we have an action by the $\mathcal K$-torus 
$\mathrm{Spec}\, \mathcal K[\mathrm{Span}(\Gamma) \cap \mathbb Z^n]$. 

\medskip
\emph{Claim}: Let $S$ be the closed subscheme of 
$\mathrm{Spec}\, R_\Gamma$ defined by the ideal of $R_\Gamma$
generated by $t_1\frac{\partial f_\Gamma}{\partial t_1},\ldots, t_n\frac{\partial f_\Gamma}{\partial t_n}$. Then the set 
$S$ is contained in the image of the section 
$$i_0: \mathrm{Spec}\,\mathcal K\to \mathcal K[\mathbb R_{\geq 0}\Gamma\cap \mathbb Z^n]$$
of the $\mathcal K$-scheme $\mathrm{Spec}\, R_\Gamma$ defined by the $\mathcal K$-epimorphism
\begin{eqnarray}\label{eqn:vertex}
\mathcal K[\mathbb R_{\geq 0}\Gamma\cap \mathbb Z^n]\to \mathcal K,\quad 
\mathbf t^{\mathbf w}\mapsto 0.
\end{eqnarray}

The closed subscheme $i_0$ is the orbit of the torus action consisting of fixed points. To prove the claim, it suffices to show 
the intersections of $S$ with other orbits of the torus action are empty. Each of the other orbits 
is isomorphic to $\mathrm{Spec}\, \mathcal K[\mathrm{Span}(\Gamma') \cap \mathbb Z^n]$ for a face $\Gamma'$ of $\Gamma$. 
The intersection of $S$ with such an orbit is the closed subscheme of defined by the ideal 
of $\mathcal K[\mathrm{Span}(\Gamma') \cap \mathbb Z^n]$
generated by $t_1\frac{\partial f_{\Gamma'}}{\partial t_1},\ldots, t_n\frac{\partial f_{\Gamma'}}{\partial t_n}$. 
If $\mathcal K=k(\mathbf a)$, this intersection is empty since $F(\mathbf t,\mathbf a)$ is nondegenerate with respect to $\Delta_\infty$. 
If $\mathcal K=\mathcal O_V(V)$, then the fiber of the intersection 
over each rational point $\mathbf a$ of $V$ is empty, and hence the whole intersection is also empty. This proves our claim. 

Note that $\mathbb Z^n/(\mathrm{Span}(\Gamma) \cap \mathbb Z^n)$ is torsion free and hence a free $\mathbb Z$-module. 
So $\mathrm{Span}(\Gamma) \cap \mathbb Z^n$ is a direct factor of $\mathbb Z^n$. Choose a basis 
$\{e'_1,\ldots, e'_n\}$ of $\mathbb Z^n$ so that $\{e'_1,\ldots, e'_{n-p}\}$ is a basis of $\mathrm{Span}(\Gamma) \cap \mathbb Z^n$.
Let $(t'_1, \ldots, t'_n)$ be the coordinate of the torus $\mathrm{Spec}\, \mathcal K[\mathbb Z^n]$
with respect to the basis $\{e'_1,\ldots, e'_n\}$. Then $f_\Gamma$ does not depend on the variables $t'_{n-p+1}, \ldots, t'_n$ 
and hence 
$$t'_{n-p+1}\frac{\partial f_{\Gamma}}{\partial t'_{n-p+1}}=\cdots= 
t'_n\frac{\partial f_{\Gamma}}{\partial t'_n}=0.$$
Over $\mathbb Z$, the set $\{t_1\frac{\partial f_\Gamma}{\partial t_1},\ldots, t_n\frac{\partial f_\Gamma}{\partial t_n}\}$
and the set $\{t'_1\frac{\partial f_\Gamma}{\partial t'_1},\ldots, t'_n\frac{\partial f_\Gamma}{\partial t'_n}\}$ span the same 
space. Thus over $\mathbb Q$, the set $\{t_1\frac{\partial f_\Gamma}{\partial t_1},\ldots, t_n\frac{\partial f_\Gamma}{\partial t_n}\}$
spans a vector space of dimension at most $n-p$. Choose a subset 
$\{t_{i_1}\frac{\partial f_\Gamma}{\partial t_{i_1}},\ldots, t_{i_{n-p}}\frac{\partial f_\Gamma}{\partial t_{n-p}}\}$
spanning the same space as the set $\{t_1\frac{\partial f_\Gamma}{\partial t_1},\ldots, t_n\frac{\partial f_\Gamma}{\partial t_n}\}$. 
Then the ideal of $R_\Gamma$ generated by these two sets are the same.

In the case $\mathcal K=k(\mathbf a)$,  $R_\Gamma$ is a Cohen-Macaulay ring of dimension $n-p$ 
by Hochster's theorem (\cite{H}). By the above claim, the maximal ideal $\mathfrak m$ defined as the kernel of the epimorphism
(\ref{eqn:vertex}) is the only prime ideal of $R_\Gamma$ containing the sequence
$t_{i_1}\frac{\partial f_\Gamma}{\partial t_{i_1}},\ldots, t_{i_{n-p}}\frac{\partial f_\Gamma}{\partial t_{n-p}}$. Hence 
the quotient of $R_\Gamma$ by the ideal generated by this sequence
has finite length. By \cite[Theorem IV.B.2]{S},   
this sequence is a regular sequence in the local ring 
$R_{\Gamma,\mathfrak m}$. By Lemma 
\ref{lm:finite} below, we have $H^q(K^\cdot(R_\Gamma))_{\mathfrak m}=0$ for $q< n-p$, and 
$H^{n-p}(K^\cdot(R_\Gamma))_{\mathfrak m}$ is finite over $\mathcal K$. By \cite[Proposition IV.A.4]{S} and the claim, 
the support of the module $H^q(K^\cdot(R_\Gamma))$ 
is contained in the set $\{\mathfrak m\}$. It follows that  $H^q(K^\cdot(R_\Gamma))=0$ for $q< n-p$ and 
$H^{n-p}(K^\cdot(R_\Gamma))$ is finite over $\mathcal K$.

In the case $\mathcal K=\mathcal O_V(V)$, $R_\Gamma$ is a Cohen-Macaulay ring of dimension $n-p+N$. 
By \cite[Proposition IV.A.4]{S}, the support of the module $H^q(K^\cdot(R_\Gamma))$ 
is contained in the closed subscheme $S$ in the above claim. Let $\mathbf a=(a_1,\ldots, a_N)$ be a rational point of $V$
so that $i_0(\mathbf a)\in S$, and 
let $\mathfrak m$ be the maximal ideal of $R_\Gamma$ corresponding to the point $i_0(\mathbf a)$. 
By the claim, the closed subscheme of $\mathrm{Spec}\,R_\Gamma$ defined by the ideal 
generated by the sequence
$$t_{i_1}\frac{\partial f_\Gamma}{\partial t_{i_1}},\ldots, t_{i_{n-p}}\frac{\partial f_\Gamma}{\partial t_{n-p}}, x_1-a_1, \ldots, x_N-a_N$$
is $i_0(\mathbf a)$ as a set. So the quotient of $R_{\Gamma}$ by the ideal generated by this sequence
has finite length,
and hence the sequence is a regular sequence in $R_{\Gamma,\mathfrak m}$. By Lemma 
\ref{lm:finite} below, we have
$H^q(K^\cdot(R_\Gamma))_{\mathfrak m}=0$ for $q< n-p$ and 
$H^{n-p}(K^\cdot(R_\Gamma))_{\mathfrak m}$ can be identified with a submodule of
$$\Big(R_\Gamma\Big/\Big(t_{i_1}\frac{\partial f_\Gamma}{\partial t_{i_1}},\ldots, t_{i_{n-p}}
\frac{\partial f_\Gamma}{\partial t_{n-p}}\Big)\Big)
_{\mathfrak m}.$$ This is true for all maximal ideal $\mathfrak m$ in the support of the Koszul cohomology. So 
$H^q(K^\cdot(R_\Gamma))=0$ for $q< n-p$, and $H^{n-p}(K^\cdot(R_\Gamma))$ is a submodule 
of $R_\Gamma\big/\big(t_{i_1}\frac{\partial f_\Gamma}{\partial t_{i_1}},\ldots, t_{i_{n-p}}
\frac{\partial f_\Gamma}{\partial t_{n-p}}\big)$ which 
is finitely generated over $\mathcal K$ by the claim. 
\end{proof}

\begin{lemma}\label{lm:finite} Let $\mathcal K$ be a commutative ring, $R$ a commutative 
$\mathcal K$-algebra,  $N$ an $R$-module, 
$f_1, \ldots, f_n$ a sequence in $R$, and $K^\cdot(f_1, \ldots, f_n,N)$ the Koszul complex for 
$N$ with respect to the sequence $f_1, \ldots, f_n$. Suppose $f_1, \ldots, f_d$ is a regular sequence on $N$
for some $d\leq n$. Then we have $H^q(K^\cdot(f_1, \ldots, f_n, N))=0$ for $q<d$ and 
$H^d(K^\cdot(f_1, \ldots, f_n, N))$ can be identified with a submodule of
of $N/(f_1, \ldots, f_d)N$. In particular, if $\mathcal K$ is noetherian and $N/(f_1, \ldots, f_d)N$ is 
a finitely generated $\mathcal K$-module, then so is $H^d(K^\cdot(f_1, \ldots, f_n, N))$. 
\end{lemma}

\begin{proof} We use induction on $n$. 
Let $e_1, \ldots, e_n$ be a basis of a free $R$-module. Recall that the Koszul complex $K^\cdot(f_1, \ldots, f_n,N)$
can be defined by
\begin{align*}
K^q(f_1, \ldots, f_n,N)&=\bigoplus_{1\leq i_1<\ldots<i_q\leq n}N e_{i_1}\wedge\cdots\wedge e_{i_q},\\
d(\omega)&= \omega\wedge (f_1e_1+\cdots+f_n e_n).
\end{align*}
We have a short exact sequence of complexes
$$0\to K^\cdot(f_1, \ldots, f_{n-1},N)[-1]\stackrel{\wedge e_n}\to K^\cdot(f_1, \ldots, f_n,N)
\to K^\cdot(f_1, \ldots, f_{n-1},N)\to 0.$$ It induces a long exact sequence
\begin{align}\label{eqn:longseq}
&\\
&\cdots\to H^{q-1}(K^\cdot(f_1, \ldots, f_{n-1},N))
\to H^q(K^\cdot(f_1, \ldots, f_n,N))\to H^q(K^\cdot(f_1, \ldots, f_{n-1},N))\to \cdots.\nonumber
\end{align}
If $d\leq n-1$, then by the induction hypothesis, 
$H^{q-1}(K^\cdot(f_1, \ldots, f_{n-1},N))$ and $H^q(K^\cdot(f_1, \ldots, f_{n-1},N))$ vanish 
for $q<d$. This implies that $H^q(K^\cdot(f_1, \ldots, f_n,N))=0$ for $q<d$, and 
$H^d(K^\cdot(f_1, \ldots, f_n,N))$ is a submodule of 
$H^d(K^\cdot(f_1, \ldots, f_{n-1},N))$. Repeating this argument, we see
$H^d(K^\cdot(f_1, \ldots, f_n,N))$ is a submodule of 
$H^d(K^\cdot(f_1, \ldots, f_d,N))\cong N/(f_1, \ldots, f_d)N$. 

It remains to 
prove the initial step $n=d$ of the induction, that is, 
$H^q(K^\cdot(f_1, \ldots, f_d,N))=0$ for $q<d$. We prove this by induction on $d$. For $d=1$, this follows from the 
assumption that $f_1: N\to N$ is injective. In general, the long exact sequence
(\ref{eqn:longseq}) with $n$ replaced by $d$ and the induction hypothesis 
$$H^q(K^\cdot(f_1, \ldots, f_{d-1},N))=0 \hbox{ for } q<d-1$$
imply that $$H^q(K^\cdot(f_1, \ldots, f_d,N))=0 \hbox{ for } q<d-1.$$
Moreover, the long exact sequence
(\ref{eqn:longseq}) with $q=d-1$ and $n=d$ can be identified with 
$$0\to H^{d-1}(K^\cdot(f_1, \ldots, f_d,N))\to N/(f_1,\ldots, f_{d-1})N
\stackrel {f_d}\to  N/(f_1,\ldots, f_{d-1})N.$$
By our assumption, $f_d: N/(f_1,\ldots, f_{d-1})N
\to  N/(f_1,\ldots, f_{d-1})N$ is injective. It follows that $H^{d-1}(K^\cdot(f_1, \ldots, f_d,N))=0$. 
\end{proof}

\begin{lemma}\label{lm:rank} Keep the notation in Lemma \ref{prop:Kouchnirenko}.
In the case where $\mathcal K=k(\mathbf a)$, we have 
$\mathrm{dim}\, H^n(K^\cdot(R))=n!\mathrm{vol}(\Delta_\infty)$. 
\end{lemma}

\begin{proof} 
For a graded $R$-module $M=\bigoplus_{d=0}^\infty M_d$, let 
$$P_M(t)= \sum_{d=0}^\infty (\mathrm{dim}\, M_d) t^d$$ 
be the Poincar\'e series. The operator $d: K^q(R)\to K^{q+1}(R)$ 
increases the degree by $M$. So $H^n(K^\cdot(R))$ is a graded $R$-module. Since $H^p(K^\cdot(R))=0$ for $p\not =n$, 
we have 
$$(-1)^n\mathrm{dim}\, H^n(K^\cdot(R))_d=\sum_{q=0}^n (-1)^q \mathrm{dim}\,K^q(R)_{d-(n-q)M}
=\sum_{q=0}^n (-1)^q  {n\choose q} \mathrm{dim}\, R_{d-(n-q)M}.$$
Therefore 
$$P_{H^n(K^\cdot(R))}(t)= P_R(t) (1-t^M)^n.$$
The resolution $A^\cdot$ of $R$ in the proof of Lemma \ref{prop:Kouchnirenko} preserves the degree. So we have 
$$P_R(t)= \sum_{q} (-1)^q P_{A^q}(t)=\sum_{q=0}^{n-1}\sum_{\Gamma\in I_{n-1-q}} (-1)^q P_{R_\Gamma}(t)$$
if $0$ lies on the boundary of $\Delta_\infty$, and we add a term $(-1)^n P_{\mathcal K}(t)$ if 
$0$ lies in the interior of $\Delta_\infty$.
A direct computation (\cite[Lemme 2.9]{K}) shows that $P_{R_\Gamma}(t)$ is a rational function of $t$, with a pole of order 
$\mathrm{dim}\,\Gamma+1$ at $t=1$, and if $\mathrm{dim}\,\Gamma=n-1$, we have
$$P_{R_\Gamma}(t)(1-t^M)^n|_{t=1}=n!\mathrm{vol}\Big(\bigcup_{t\in [0,1]}t\Gamma\Big).$$
So we have
\begin{align*}
&\mathrm{dim}\, H^n(K^\cdot (R))=P_{H^n(K^\cdot(R))}(t)|_{t=1}
= P_R(t) (1-t^M)^n|_{t=1}\\
&=\sum_{q=0}^{n-1}\sum_{\Gamma\in I_{n-1-q}} (-1)^q P_{R_\Gamma}(t)(1-t^M)^n|_{t=1}
=\sum_{\Gamma\in I_{n-1}} n!\mathrm{vol}\Big(\bigcup_{t\in [0,1]}t\Gamma\Big)
=n!\mathrm{vol}(\Delta_\infty)
\end{align*}
if $0$ lies on the boundary of $\Delta_\infty$. If $0$ lies in the interior of $\Delta_\infty$, the result is the same since we have 
$P_{\mathcal K}(t)(1-t^M)^n|_{t=1}=0$. 
\end{proof}

\subsection{Proof of Proposition \ref{prop:goodR}} 
Let $V$ be an arbitrary affine open subset of $U$.
We have $\Gamma(V,\mathcal H^i(\mathcal C_{\boldsymbol\gamma}))\cong H^i(C^\cdot)$ with $\mathcal K=\mathcal O_V(V)$. By Lemma \ref{lm:finiteD}, $H^i(C^\cdot)=0$ and $H^n(C^\cdot)$ is a finitely generated $\mathcal O_V(V)$-module.
So $\mathcal H^i(\mathcal C^\cdot_{\boldsymbol\gamma})|_V=0$ for $i\not=n$, and 
$\mathcal H^n(\mathcal C^\cdot_{\boldsymbol\gamma})|_V$ is a coherent 
$\mathcal O_V$-module, and hence an integrable connection. Let 
$\mathbf a$ be a rational point on $U$, and let $i_{\mathbf a}: \mathbf a\to \mathbb A^N$ be the closed immersion. 
Since $\mathcal C^\cdot_{\boldsymbol\gamma}$ is a complex of flat $\mathcal O_{\mathbb A^N}$-modules
and each $d: \mathcal C^q_{\boldsymbol\gamma}\to \mathcal C^{q+1}_{\boldsymbol\gamma}$ is 
$\mathcal O_{\mathbb A^N}$-linear, we have 
$i^*_{\mathbf a} \mathcal H^n(\mathcal C^\cdot_{\boldsymbol\gamma})
\cong H^n(i^*_{\mathbf a} \mathcal C^\cdot_{\boldsymbol\gamma})$. Note that
$i^*_{\mathbf a} \mathcal C^\cdot_{\boldsymbol\gamma}$ is represented by the complex $C^\cdot$ with 
$\mathcal K=k(\mathbf a)$. By Lemma \ref{lm:finiteD}, the dimension of the fiber at $\mathbf a$ of the integrable connection 
$\mathcal H^n(\mathcal C^\cdot_{\boldsymbol\gamma})|_V$ is $n!\mathrm{vol}(\Delta_\infty)$. \qed

\end{document}